\documentclass[12pt]{amsart}
\usepackage{amssymb}

\usepackage{hyperref}
\usepackage[all]{hypcap}
\usepackage[english]{babel}
\usepackage[utf8]{inputenc}
\usepackage{graphics}
\usepackage{verbatim}
\usepackage{listings}
\usepackage[mathcal]{eucal}
\usepackage{tikz}
\usepackage{tikzit}
\usepackage{enumerate}
\usepackage[left=3.8cm, right=3.8cm]{geometry}
\usepackage{ifthen}

\newboolean{colour}
\setboolean{colour}{true}

\ifthenelse{\boolean{colour}}{% TiKZ style file generated by TikZiT. You may edit this file manually,
\tikzstyle{red_f}=[fill=red, draw=black, shape=rectangle, tikzit category={functional_nodes}]
\tikzstyle{green_f}=[fill=green, draw=black, shape=rectangle, tikzit category={functional_nodes}]
\tikzstyle{dark_blue3_f}=[fill={rgb,255: red,8; green,71; blue,125}, draw=black, shape=rectangle, tikzit category={functional_nodes}]
\tikzstyle{light_blue3_f}=[fill={rgb,255: red,207; green,226; blue,243}, draw=black, shape=rectangle, tikzit category={functional_nodes}]
\tikzstyle{dark_cornflower_blue2_f}=[fill={rgb,255: red,17; green,85; blue,204}, draw=black, shape=rectangle, tikzit category={functional_nodes}]
\tikzstyle{red_nf}=[fill=red, draw=black, shape=circle, tikzit category=nodes]
\tikzstyle{green_nf}=[fill=green, draw=black, shape=circle, tikzit category=nodes]
\tikzstyle{dark_orange2_nf}=[fill={rgb,255: red,180; green,95; blue,6}, draw=black, shape=circle, tikzit category=nodes]
\tikzstyle{magenta_nf}=[fill=magenta, draw=black, shape=circle, tikzit category=nodes]
\tikzstyle{new style 0}=[fill={rgb,255: red,153; green,0; blue,255}, draw=black, shape=circle]
\tikzstyle{dark_yellow1}=[fill={rgb,255: red,241; green,194; blue,50}, draw=black, shape=circle, tikzit category=nodes, minimum size=6mm]
\tikzstyle{dark_yellow2_nf}=[fill={rgb,255: red,191; green,144; blue,0}, draw=black, shape=circle, tikzit category=nodes, minimum size=6mm]
\tikzstyle{dark_cornflower_blue2_nf}=[fill={rgb,255: red,17; green,85; blue,204}, draw=black, shape=circle, tikzit category=nodes, minimum size=6mm]
\tikzstyle{dark_green1_nf}=[fill={rgb,255: red,106; green,168; blue,79}, draw=black, shape=circle, tikzit category=nodes, minimum size=6mm]
\tikzstyle{dark_green2_nf}=[fill={rgb,255: red,56; green,118; blue,29}, draw=black, shape=circle, tikzit category=nodes, minimum size=6mm]
\tikzstyle{cornflower_blue_nf}=[fill={rgb,255: red,74; green,134; blue,232}, draw=black, shape=circle, tikzit category=nodes, minimum size=6mm]
\tikzstyle{light_blue3_nf}=[fill={rgb,255: red,207; green,226; blue,243}, draw=black, shape=circle, tikzit category=nodes, minimum size=6mm]

\tikzstyle{dashed}=[-, dotted]
\tikzstyle{new edge style 0}=[-, new atom]
}{% TiKZ style file generated by TikZiT. You may edit this file manually,
\tikzstyle{red_f}=[fill=white, draw=black, shape=rectangle, tikzit category={functional_nodes}]
\tikzstyle{green_f}=[fill=white, draw=black, shape=rectangle, tikzit category={functional_nodes}]
\tikzstyle{dark_blue3_f}=[fill=white, draw=black, shape=rectangle, tikzit category={functional_nodes}]
\tikzstyle{light_blue3_f}=[fill=white, draw=black, shape=rectangle, tikzit category={functional_nodes}]
\tikzstyle{dark_cornflower_blue2_f}=[fill=white, draw=black, shape=rectangle, tikzit category={functional_nodes}]
\tikzstyle{red_nf}=[fill=white, draw=black, shape=circle, tikzit category=nodes]
\tikzstyle{green_nf}=[fill=white, draw=black, shape=circle, tikzit category=nodes]
\tikzstyle{dark_orange2_nf}=[fill=white, draw=black, shape=circle, tikzit category=nodes]
\tikzstyle{magenta_nf}=[fill=white, draw=black, shape=circle, tikzit category=nodes]
\tikzstyle{new style 0}=[fill=white, draw=black, shape=circle]
\tikzstyle{dark_yellow1}=[fill=white, draw=black, shape=circle, tikzit category=nodes, minimum size=6mm]
\tikzstyle{dark_yellow2_nf}=[fill=white, draw=black, shape=circle, tikzit category=nodes, minimum size=6mm]
\tikzstyle{dark_cornflower_blue2_nf}=[fill=white, draw=black, shape=circle, tikzit category=nodes, minimum size=6mm]
\tikzstyle{dark_green1_nf}=[fill=white, draw=black, shape=circle, tikzit category=nodes, minimum size=6mm]
\tikzstyle{dark_green2_nf}=[fill=white, draw=black, shape=circle, tikzit category=nodes, minimum size=6mm]
\tikzstyle{cornflower_blue_nf}=[fill=white, draw=black, shape=circle, tikzit category=nodes, minimum size=6mm]
\tikzstyle{light_blue3_nf}=[fill=white, draw=black, shape=circle, tikzit category=nodes, minimum size=6mm]

\tikzstyle{dashed}=[-, dotted]
\tikzstyle{new edge style 0}=[-, new atom]
}

\title[Functional reducts of the countable atomless Boolean algebra]{Functional reducts of the countable atomless Boolean algebra}

\author[B. Bodor]{Bertalan Bodor}
\address[1,2,3]{E\"{o}tv\"{o}s Lor\'{a}nd University, 
          Department of Algebra and Number Theory,
         Hungary, 1117 Budapest,  P\'{a}zm\'{a}ny P\'{e}ter s\'{e}t\'{a}ny 1/c}
\email{bodorb@cs.elte.hu}

\author[K. Kalina]{Kende Kalina}
\email{kkalina@cs.elte.hu}

\author[Cs. Szab\'{o}]{Csaba Szab\'{o}}
\email{csaba@cs.elte.hu}

\newcommand{\mysize}{\normalsize}

\newtheorem{theorem}{Theorem}
\newtheorem{definition}[theorem]{Definition}
\newtheorem{lemma}[theorem]{Lemma}
\newtheorem{corollary}[theorem]{Corollary}
\newtheorem{problem}{Problem}

\DeclareMathOperator{\Aut}{Aut}

    \newcounter{csopk}
    \newenvironment{csop}[1]{\refstepcounter{csopk}\par\medskip\noindent%
       \textbf{Group \arabic{csopk}: {#1}} \par\addvspace{\baselineskip} \rmfamily}{\medskip}

\begin{document}
\newcommand{\nc}{\newcommand}
\nc{\ok}{$\omega$-categorical}

\nc{\ff}{\mathcal{F}}
\nc{\ee}{\alpha}
\nc{\cc}{\mathcal{C}}
\nc{\rr}{\mathcal{R}}
\nc{\on}{\operatorname}
\nc{\aut}[1]{\on{Aut}\left({#1}\right)}
\nc{\zj}[1]{\left({#1}\right)}
\nc{\ba}{\mathfrak{Ba}}
\nc{\sym}{\on{Sym}(\ba)}
\nc{\aba}{\Aut(\ba)}
\nc{\eps}{\vec{\varepsilon}}
\nc{\al}{\alpha}
\nc{\thi}{\tilde{\phi}}
\nc{\nb}{\nsubseteq}

\nc{\mxi}{\mathfrak{I}}
\nc{\theo}{\mathcal{T}}
\nc{\simi}{\on{Sym} \left( \mxi \right)}
\nc{\simik}{\on{Sym}_{ \{ \mxi \} } \left( \ba \right)}

\nc{\ga}{\mathfrak{A}}
\nc{\gb}{\mathfrak{B}}
\nc{\mbn}{\mathbb{N}}
\nc{\cl}[1]{\operatorname{Cl} \left< {#1} \right>}

\renewcommand{\phi}{\varphi}

\nc{\Aref}{\ref}
\nc{\aref}{\ref}
\begin{abstract}

For an algebra $\mathfrak{ A}=(A,f_1,\dots, f_n)$ the algebra  $\mathfrak{B}=(A,t_1, \ldots, t_k)$ is called a functional reduct if each $t_j$ is a term function of $\mathfrak{A}$. We classify the \emph{functional reducts} of the countable atomless Boolean algebra up to first-order interdefinability. That is, we consider two functional reducts the ``same", if their group of automorphisms is the same.  We show that there are 13 such reducts and describe their structures and group of automorphisms.
\end{abstract}

\maketitle

\section{Introduction}
%Jelölje $\ba=(B,\wedge,\vee,0,1,\neg)$ a megszámlálható atommentes Boole-algebrát. Izomorfia
%erejéig pontosan egy  ilyen
%struktúra létezik. Könnyű ellenőrizni, hogy $\ba$ bármely két véges
%részalgebrája között menő
%izomorfizmus kiterjed $\ba$ egy automorfizmusává, továbbá
%minden véges vagy
%megszámlálhatóan végtelen Boole-algebra
%beágyazható $\ba$-ba. Azaz $\ba$ homogén és univerzális a
%Boole-algebrák osztályában.

Let  $\ba=(B,\wedge,\vee,0,1,\neg)$ denote the countable atomless Boolean algebra. It is known that this structure is
unique up to isomorphism. It is easy to check that the structure $\ba$ is \emph{homogeneous}, that is every isomorphism between finite substructures of $\ba$ can be extended
to an automorphism of $\ba$. %, and every countable Boolean algebra can be embedded into $\ba$. 
%So $\ba$ is a homogeneous and universal object in the class of countable Boolean algebras.
The structure $\ba$ is also \emph{{\ok}} which means that every countable structure with the same first-order theory is isomorphic to $\ba$.
%A $\ba$ struktúra
%előáll, mint a véges
%Boole-algebrák osztályának Fra\"issé-limesze \cite{frass}.
The Boolean algebra $\ba$ can be obtained as the Fra\"iss\'{e} limit of the class of finite Boolean algebras.
Homogeneous structures have been extensively studied and they are often interesting on their own. Some of the classical examples for homogeneous structures are the random graph, the countably infinite dimensional vector spaces over finite fields, and the rationals as an ordered set. %Another simple example is the family of Henson graphs \cite{henson}, and the homogeneous
%partially ordered set, which can be obtained as a result of a random process like the well-known random graph \cite{velpos} 
	The general theory of homogeneous structures was initiated by Fra\"iss\'{e} \cite{frass}. %Fra\"iss\'{e}'s theorem characterizes those classes of finitely generated structures which can be obtained as the finitely generated
%substructures of a homogeneous structure.  
 We refer the reader to \cite{macpherson} for a detailed discussion about homogeneous structures.

\begin{definition}
%Egy ${\mathbf A}=(A,f_1,\dots,f_n)$ algebra egy funkcionális reduktján
%egy $(A,t_1,\dots, t_2)$ algebrát értünk, ahol a $t_i$ kifejezések
%előállnak az $f_i$-k iterált kompozícióiként, azaz a $t_i$ műveletek a
% $\mathbf A$ algebra kifejezés-függvényei.
Let $\mathfrak{ A}=(A,f_1,\dots, f_n)$ be an algebra on the set $A$ with operations $f_1, \ldots, f_n$. A \emph{functional reduct} of
$\mathfrak{A}$ is a structure $\mathfrak{B}=(A,t_1, \ldots, t_k)$ on the same set $A$ and with operations $t_1, \ldots, t_k$
such that every $t_j$ is a term function of $\mathfrak{A}$ (see Definition \ref{def_term}).
\end{definition}

%Például $\mathfrak{A}=(A, \Delta)$, ahol $\Delta$
%a szimetrikus differenciát jelöli, $\ba$ egy Abel-csoport reduktja.
%Ebben a dolgozatban $\ba$ funkcionális reduktjait klasszifikáljuk. Megmutatjuk, hogy lényegében 13 ilyen funkcionális redukt létezik, és
%jellemezzük ezen funkcionális reduktokat.
For example the structure $\mathfrak{A}=(\ba, \Delta)$ where $\Delta$ denotes the symmetric difference, $\Delta(x,y)= \neg( x\wedge y)\ \wedge (x\vee y) $ is a functional reduct
of $\ba$ which is an Abelian group. In this paper we will classify the functional reducts of $\ba$.
We will show that there are exactly $13$ of them up to first-order interdefinability. The example given above is denoted by $(+_0,0)$ later in the paper.

%A  $\ba$ algera \ok.
%Egy $\ga$ struktúrát \ok nak nevezünk, ha megszámlálható, és
%elméletének minden megszámlálható modellje izomorf $\ga$-val.
%Azaz, ha egy $\ga$ struktúrában ugyanazok az elsőrendű formulák
%igazak, mint $\ba$-ban,
%akkor $\ga$ izomorf kell legyen $\ba$-val.

%Ebben a dolgozatban egy struktúra egy reduktja alatt a struktúra alaphalmazán értelmezett relációk olyan halmazát értjük, amelyek
%mindegyikére
%igaz, hogy elsőrendű formulákkal definiálható a struktúrában.
%A reduktok halmazán értelmezhető egy kvázirendezés: $R_1 \lesssim R_2$ pontosan akkor, ha
%$R_2$ minden relációja definiálható $R_1$ feletti elsőrendű formulákkal. Két redukt kölcsönösen definiálható egymással, ha
%$R_1 \lesssim R_2$ és $R_2 \lesssim R_1$,
%azaz $R_1$ minden relációját definiálni lehet $R_2$ feletti elsőrendű formulákkal, és $R_2$ minden relációját definiálni lehet
%$R_1$ feletti elsőrendű formulákkal. \Aref{ryll}. Tétel fontos következménye, hogy {\ok} struktúra minden reduktja is {\ok}.
%Ez nem {\ok} struktúrákra még akkor sem feltétlen igaz, ha a struktúra nyelve véges, és csak relációkat tartalmaz. A
%\cite{thomas}-ban megtalálható
%Lachlan egy ellenpéldájának leírása.
If $\mathfrak{A}$ is a first-order structure then we will call a structure $\mathfrak{B}$ a \emph{reduct} of $\mathfrak{A}$ if $\mathfrak{B}$ has the same domain set as $\mathfrak{A}$, and all constants, relations, and functions of $\mathfrak{B}$ are first-order definable over $\mathfrak{A}$.
% where all $r_j$-s are first-order definable relations over $\mathfrak{A}$ as a reduct of $\mathfrak{A}$. 
For a given first-order structure
we can define a preorder on its reducts: $\mathfrak{B}_1 \lesssim \mathfrak{B}_2$ if and only if $\mathfrak{B}_2$ is a reduct of $\mathfrak{B}_1$. Two reducts are called first-order interdefinable if and only if $\mathfrak{B}_1 \lesssim \mathfrak{B}_2$ and $\mathfrak{B}_2 \lesssim \mathfrak{B}_1$.
It is an important consequence of Theorem \ref{ryll} that a reduct of an {\ok} structure is also {\ok} itself.

%Egy redukt automorfizmuscsoportja pontosan azon permutációkból
%áll, amelyek megőrzik a redukt összes relációját, így speciálisan tartalmazza az eredeti struktúra automorfizmus-csoportját.
%Továbbá, egy redukt automorfizmus-csoportjára teljesül a következő zártsági feltétel: ha $g_1, g_2 \ldots \in \aut{R}$ permutációk
%egy sorozata a redukt automorfizmus-csoportjából, amelyekre teljesül, hogy a struktúra minden $a$ elemére van olyan $j$
%index és $b_a$ elem, hogy
%minden $n>j$ indexre $g_n(a)=b_a$, akkor a $g_1,g_2 \ldots$ sorozat határértéke, a $h(x)=b_x$ permutáció is benne van az automorfizmuscsoportban.

The automorphism group of a reduct consists of the permutations preserving all relations of the reduct, so the automorphism group
of a reduct contains the automorphism group of the original structure. In general, if $\mathfrak{B}_1 \lesssim \mathfrak{B}_2$ then
$\Aut(\mathfrak{B}_1) \leq \Aut(\mathfrak{B}_2)$. Moreover, %the following condition holds for the automorphism group of every reduct:
%if $g_1, g_2, \ldots \in \Aut(\mathfrak{B})$ are permutations such that
%for every element $a \in A$ there exists an index $j$ and an element $b_a$ such that
%for all indices $n>j$ the equality $g_n(a)=b_a$ holds,
%then the limit of the sequence $g_1,g_2, \ldots$ the permutation $h(x)=b_x$ is also in $\Aut(\mathfrak{B})$.
%In other words 
the group $\Aut(\mathfrak{B})$ is closed in the topology of pointwise convergence, i.e.\ the subspace topology of the product topology on $A^A$ where %the domain set of $\mathfrak{B}$ 
$A$ is equipped with the discrete topology.

%Azonban {\ok} struktúrák esetén ennél erősebb is igaz: a struktúra automorfizmus-csoportját tartalmazó zárt részcsoportok
%bijekcióban állnak a reduktok kölcsönös definiálhatóságra vett ekvivalenciaosztályaival,
%így {\ok} struktúrák esetén a reduktok klasszifikálása ekvivalens a struktúra automorfizmus-csoportját tartalmazó zárt
%részcsoportok osztályozásával.

If the structure $\mathfrak{A}$ is {\ok} then there is a bijection between the closed subgroups of $\on{Sym}(A)$
containing $\Aut(\mathfrak{A})$ and the equivalence classes of reducts of $\mathfrak{A}$ by first-order interdefinability.
So in the case of {\ok} structures the classification of the reducts is equivalent with the description of the closed groups
containing the automorphism group. %This statement is not true for non {\ok} structures, even if we restrict our attention to structures with finite purely relational
language. 
%There is a counterexample of Lachlan presented in \cite{thomas}.

%Számos {\ok} homogén struktúrának sikerült osztályozni a reduktjait kölcsönös definiálhatóság erejéig.
%Kölcsönös definiálhatóság erejéig öt különböző reduktja van például a racionális számoknak a szokásos rendesékükkel
%ellátva \cite{camren},
%a véletlen gráfnak \cite{thomas}, a véletlen turnamentnek \cite{benn} és a véletlen poszetnek \cite{pppsz}. A klasszifikáció ismert a konstanssal ellátott
%Henson-grafokra is \cite{pongi}.
%A homogén rendezett gráfnak viszont már
%több, mint $40$ \cite{nk}, a racionális számoknak a rendezéssel és egy konstanssal ellátva már $116$ \cite{juzi}, a véletlen gráfnak egy konstanssal
%ellátva pedig már több, mint 300 reduktja van. Ez utóbbira nem is ismert a teljes klasszifikáció. Ennek alapján megfogalmazható, hogy minél több eleme van a nyelvnek,
%várhatóan annál nagyobb lesz a reduktok száma, és ez egyszerűnek tűnő struktúrákra is meglepően nagy tud lenni.
%Mivel a $\ba$ nyelve több jelet tartalmaz az itt felsorolt struktúrák nyelveinél, így itt is sok reduktra lehet számítani.

The classification of the reducts up to first-order interdefinability is known for a handful of {\ok} homogeneous structures.
For example the set of rationals equipped with the usual ordering has five reducts up to first-order interdefinability
\cite{camren}, and similarly the the random graph \cite{thomas}, the random tournament \cite{benn} and the random partially ordered set
\cite{pppsz} also have five. The classification is also known for the Henson graphs \cite{pongi}.
The homogeneous ordered graph has more than $40$ \cite{nk}, and $\mathbb{Q}^{\leq}$ (rationals with the usual ordering) equipped with an additional constant has
$116$ %, and the random graph equipped with an additional constant has more than $300$ 
reducts \cite{juzi}. %In the case of the last example,
%the complete classification is not known. 
	Note that in all these classifications the languages of the structures in question do not contain any function symbols.
%So we can conclude that structures with larger languages generally have more reducts,
%and seemingly simple-looking  structures can have many reducts up to first-order interdefinability. Since the language of $\ba$
%has five symbols we can expect the existence of many reducts.

%In all the previous classifications the languages of the structures in question do not contain any function symbols. To handle these structures,
%Bodirsky and Pinsker developed some Ramsey-theoretic methods.
%The structure $\ba$ is not homogeneous over any finite relational language, so the methods used in the previous classifications
%cannot be applied.

In order to handle the reducts of these structures Bodirsky and Pinsker developed some techniques using Ramsey theory and canonical functions (see \cite{rams} for instance). However, these techniques only work for homogeneous structures over finite relational languages. It follows from an easy orbit counting argument that the structure $\ba$ is not homogeneous over any finite relational language: $\Aut(\ba)$ has at least $2^{2^{n-1}}$ $n$-orbits, whereas the automorphism group of a homogeneous structure over a finite relational language has at most $2^{p(n)}$ $n$-orbits for some polynomial $p$. Therefore the methods mentioned above cannot be applied in our case.

	In this paper we show that $\ba$ has exactly 13 functional reducts. 11 out of these 13 reducts are reducts of $\ba$ considered as a vector space together with the constant 1 (later denoted by $(+_0,0,1)$). Apart from these there is one nontrivial reduct. This reduct can be represented as $\ba$ with the median relation $M(x,y,z)=(x\vee y)\wedge (y\vee z)\wedge (z\vee x)$. The automorphism group of this reduct is generated by the translations $t_c: x\mapsto x+c$ over $\Aut(\ba)$.

	We also give a list of all reducts known to us at the end of this article, the lattice of these reducts can be seen on Figure \ref{figure:all}. This list contains 12 infinite ascending chains, and 23 sporadic reducts. Note that our list might be incomplete, %, and we have not managed to prove this in general. 
	but we hope that this classification of functional reducts can be used for the classification of all reducts in the future. 

%Az általunk vizsgált funkcionális reduktokon hasonlóan értelmezhető egy kvázirendezés:
%$\cc_1 \lesssim \cc_2$ pontosan akkor, ha
%$\cc_2$ minden művelete definiálható $\cc_1$ feletti elsőrendű formulákkal.
%Ezekre is igaz, hogy automorfizmus-csoportjuk zárt,
%és {\ok} esetben ha két funkcionális redukt automorfizmus-csoportja megegyezik,
%akkor azok elsőrendű formulákkal kölcsönösen definiálhatóak egymásból.
%Azonban a funkcionális reduktok ekvivalenciaosztályai nem állnak bijekcióban a zárt csoportokkal:
%létezhetnek olyan zárt, a struktúra automorfizmus-csoportját tartalmazó csoportok,
%amelyek nem állnak elő funkcionális redukt automorfizmus-csoportjaként.

\section{Background}

We can define a preorder similarly on the functional reducts of a given algebra: $\cc_1 \lesssim \cc_2$ if and only if all functions in $\cc_2$ can be defined over $\cc_1$
using first-order formulas. The automorphism groups of such functional reducts will also be closed. It remains true that two functional reduct is first-order
interdefinable if and only if their automorphism groups are the same. But the first-order equivalence classes of functional reducts are not in a bijection
with the closed groups: there may exist closed groups containing the automorphism group of the original structure which cannot be obtained as the automorphism
group of a functional reduct.

%Egy algebra felett előfordulhat, hogy létezik olyan függvény, amely definiálható az
%algebra műveleteivel elsőrendű formulákkal, de nem kifejezésfüggvény.
%Ilyenre példa, ha egy háló alaphalmazát tekintjük, amelyen csak a $\wedge$ művelet adott: a $\vee$ művelet definiálható elsőrendű formulával:
%$a \vee b=c \Leftrightarrow (\forall d)((d \wedge a = a$ \textit{és}  $d \wedge b=b) \rightarrow d \wedge c=c)$, viszont nem kifejezésfüggvény.
%Tehát két különböző $\cc_1$ és $\cc_2$ funkcionális reduktnak lehet azonos automorfizmus-csoportja.

\begin{definition}\label{def_term}
	Let $\mathfrak{A}=(A,f_1,\dots,f_n)$ be an algebra. Then we say that a function $g: A^k\rightarrow A$ is \emph{first-order definable} over $\mathfrak{A}$ if the relation $R:=\{(x,g(x)): x\in A^k\}$ is definable by a first-order formula using the functions $f_1,\dots,f_n$. 

	We say that a function $g$ is a \emph{term function} of the algebra $\mathfrak{A}$ if $g$ can be generated as a function from the functions $f_1,\dots,f_n$, and the projection maps $\pi_i^k: (x_1,\dots,x_k)\mapsto x_i$. 
\end{definition}

Note that not every first-order definable function is a term function: if we consider any lattice as a $\wedge$-semilattice, then the operation $\vee$ is first-order definable (see Lemma \ref{ketvalt}), but it is not a term function in the $\wedge$-semilattice.

	Let $\mathfrak{B}=(B,\wedge,\vee,0,1,\neg)$ be an arbitrary Boolean algebra. Then we define the operations $+$ and $\cdot$ on $\mathfrak{B}$ as $x+y:=\Delta(x,y)=\neg (x\wedge y)\wedge (x\vee y)$, and $x\cdot y=xy=x\wedge y$. It is well-known that in this case the structure $\mathfrak{B}$ with the operations $+$ and $\cdot$, and the constants 0 and 1 form a ring in a natural way. It is also known that this ring satisfies the identities $x^2=x$ and $2x=0$. Using this definition the algebras $\mathfrak{B}=(B,\wedge,\vee,0,1,\neg)$ and $\mathfrak{B}'=(B,+,\cdot,0,1)$ are term equivalent meaning that a function is a term function over $\mathfrak{B}$ if and only if it is a term function over $\mathfrak{B}'$. In order to see this it is enough to show that every function of $\mathfrak{B}'$ is term function of $\mathfrak{B}$, and vice versa. The former follows directly from the definition of $\mathfrak{B}'$. For the other direction we can use the following definitions: $$x\wedge y:=xy,\, x\vee y:=xy+x+y,\, \neg(x):=x+1.$$ This observation also implies that the structures $\mathfrak{B}$ and $\mathfrak{B}'$ are interdefinable.

	In most of our computations in this paper we will consider $\ba$ as a ring as above. The argument above shows this does not change what is a term function of $\ba$ or what is first-order definable over $\ba$.

%Az {\ok} struktúrák elméletében fontos szerepet tölt be az alábbi tétel \cite{hodges}:

The following theorem is essential in the theory of {\ok} structures:

\begin{theorem}[Engeler, Ryll-Nardzewski, Svenonius]\label{ryll}
%Egy $\ga$ struktúra pontosan akkor {\ok} ha $\aut{\ga}$ oligomorf. 
The structure $\ga$ is {\ok} if and only if $\Aut(\ga)$ is \emph{oligomorphic}, i.e. for all $n$ it has finitely many $n$-orbits.
\end{theorem}

The following two corollaries are easy consequences of Theorem \ref{ryll}.

\begin{corollary}\label{omkat}
%Legyen $\ga$ egy {\ok} struktúra és $R$ egy olyan reláció $\ga$ alaphalmazán, amelyet minden $\aut{\ga}$-beli permutáció megőriz.
%Ekkor $R$ definiálható egy $\ga$ feletti elsőrendű formulával.
Let $\ga$ be an {\ok} structure and $f$ be a function or a relation on the universe of $\ga$ such that every permutation of $\Aut(\ga)$ preserves $f$.
Then $f$ is first-order definable over $\ga$.
\end{corollary}

\begin{corollary}\label{omkat2}
%Legyen $\cc,\cc_1,\cc_2\ldots \cc_k$ néhány funkcionális reduktja $\ba$-nak. Ekkor a következő két feltétel ekvivalens:
Let $\cc,\cc_1,\cc_2,\ldots, \cc_k$ be some functional reducts of an algebra $\mathfrak{A}$. Then the following two conditions are equivalent:
\begin{itemize}
\item $\Aut(\cc_1) \cap \Aut(\cc_2) \cap \ldots \cap \Aut(\cc_k)=\Aut(\cc)$
\item All functions of $\cc_1 \cup \cc_2 \cup \ldots \cup \cc_k$ can be defined from $\cc$, and
 all functions of $\cc$ can be defined from $\cc_1 \cup \cc_2 \ldots \cup \cc_k$. 
\end{itemize}
\end{corollary}

	In this paper we will classify the functional reducts of the countable atomless Boolean algebra $\ba$ up to first-order interdefinability. Since
$\ba$ is {\ok}, it is sufficient to describe the closed groups $\aba \leq G \leq \sym$ which can be obtained as the automorphism group of a functional reduct.
	
	We use the following classical results in universal algebra.

\begin{theorem}\label{identities}
	Let $\mathfrak{A}$ be an algebra in which some identity $f_1(x_1,\dots,x_k)=f_2(x_1,\dots,x_k)$ is satisfied for some term functions $f_1$ and $f_2$ of $\mathfrak{A}$. Let $\mathfrak{B}$ be an algebra in the variety generated by $\mathfrak{A}$, and let $f_i^{\mathfrak{B}}$ denote the term function of $\mathfrak{B}$ which has the same definition as $f_i$ in $\mathfrak{A}$. Then the identity $f_1^{\mathfrak{B}}(x_1,\dots,x_k)=f_2^{\mathfrak{B}}(x_1,\dots,x_k)$ is satisfied in $\mathfrak{B}$.
\end{theorem}

	The subalgebra generated by the element $0,1\in\ba$ is isomorphic to the two-element Boolean algebra. We denote this subalgebra by $\mathfrak{B}_2$.

\begin{theorem}\label{boolean_variety}
	 The variety generated by $\mathfrak{B}_2$ is the class of all Boolean algebras. In particular the algebra $\ba$ itself is contained in this variety.
\end{theorem}

	We refer the reader to \cite{burris} for the proofs of theorems above, as well as a more detailed discussion about universal algebra, and varieties.

	Theorems \ref{identities} and \ref{boolean_variety} have the following consequence for term functions of $\ba$.

\begin{lemma}\label{rest}
	Let $f_1,f_2$ be term functions of $\ba$ satisfying the identity $f_1=f_2$ on $\mathfrak{B}_2$. Then $f_1=f_2$.
\end{lemma}

\begin{proof}
	Since $\mathfrak{B}_2$ is a subalgebra of $\ba$ it follows that $f_i=(f_i|_{\mathfrak{B}_2})^{\ba}$. Then Theorems \ref{identities} and \ref{boolean_variety} imply that the identity $f_1=f_2$ is satisfied in $\ba$.
\end{proof}

	Finally we provide some justification why this restricted set of reducts might be interesting for the structure $\ba$.

\begin{definition}
Let $\mathfrak{A}$ be a structure on the set $A$, let $\Aut(\mathfrak{A})$ denote its automorphism group and let $S \subset A$ be a finite subset. The \emph{group theoretic algebraic closure} of $S$ is:
$$
\operatorname{ACL}(S)=\{x \in A : \mbox{ the orbit of } x \mbox{ under the stabilizer of } {S} \mbox{ is finite}\}
$$
\end{definition}
We follow the notation of \cite[notation introduced before Lemma 4.1.1]{hodges}. It is known that if $\mathfrak{A}$ is {\ok} then for finite subsets the above definition coincides with the usual model-theoretic algebraic closure \cite{hodges}. 

	We observed that every reduct of $\ba$ that we know is similar to one of the functional reducts in the following sense: for every reduct $\rr$ there is a unique functional reduct $\ff$ such that the algebraic closure operator on $\rr$ and on $\ff$ coincides.

\section{The non-linear functional reducts}

%A dolgozat további részében a $\ba$ megszámlálható atommentes Boole-algebra funkcionális reduktjait fogjuk klasszifikálni
%kölcsönös definiálhatóság erejéig.
%Mivel ez a sruktúra {\ok}, így elegendő azon $\aba \leq G \leq \sym$ csoportokat meghatározni, amelyek előállnak valamilyen funkcionális redukt
%automorfizmus-csoportjaként.
%A struktúrával mint Boole-gyűrűvel fogunk dolgozni, mivel a Boole-algebra és a Boole-gyűrű kifejezésfüggvényei megegyeznek,
%ez megengedett. Az elsőrendű kölcsönös definiálhatóság önmagában nem lenne elégséges, de a Boole-algebra és a Boole-gyűrű
%kifejezésfüggvényei megegyeznek. 

%Ebben a fejezetben a $\ba$ megszámlálható atommentes Boole-algebra nemlineáris funkcionális reduktjait klasszifikáljuk
%Ehhez meghatározzuk a $\sym$-nak a $\aba$-t tartalmazó olyan zárt részcsoportjait,
%melyek megőriznek valamilyen nemlineáris kifejezésfüggvényt.

In this section we will deal with the non-linear functional reducts. To do so, we determine all closed subgroups of $\sym$ containing $\aba$ which preserve
some non-linear term function of $\ba$.

%First we show that every term function of $\ba$ is definable from $\wedge$ or $\vee$ alone.

	We start with the following important observation.

\begin{lemma}\label{metszet_eleg}
	The structure $\ba=(B,\wedge,\vee,0,1,\neg)$ is first-order interdefinable with the structures $(B,\wedge)$ and $(B,\vee)$.
\end{lemma}

\begin{proof}
	We only show that $\ba=(B,\wedge,\vee,0,1,\neg)$ and $(B,\wedge)$ interdefinable, the other claim is analogous. For this it is enough to show that the functions $\vee,\neg$ and the constants $0,1$ have are all first-order definable in $(B,\wedge)$. 

	Using the usual notation in lattice theory we know that $x\leq y$ if and only if $x\wedge y=x$, and $z=x\vee y$ if and only if $z$ is the smallest element with $x\leq z$ and $y\leq z$. Then we have $z=x\vee y$ if and only if $$xz=x \wedge yz=y \wedge \forall w\bigl((xw=x \wedge yw=y)\rightarrow zw=z\bigr).$$ Therefore the relation $\vee$ is first-order definable in $(B,\wedge)$. Then we can define $0,1$ as $x=0\Leftrightarrow \forall z(zx=x)$ and $x=1\Leftrightarrow \forall z(zx=z)$. Finally, we can define $\neg$ as $x=\neg (y)\Leftrightarrow (xy=0) \wedge (x\vee y=1)$.
\end{proof}

%Legyen $f(x,y)$ egy kétváltozós kifejezésfüggvény. Ekkor $f(x,y)$ felírható az $1,x,y,xy$ monomok közül valahány összegeként,
%mivel minden Boole-gyűrű kommutatív és minden elem idempotens a szorzásra nézve.

Let $f(x,y)$ be a binary term function of $\ba$. Then $f(x,y)$ can be expressed as the sum of some monomials from the set $\{1,x,y,xy\}$.

\begin{lemma}\label{ketvalt}
%Ha $\aba \leq G \leq \sym$ megőriz valamilyen $f(x,y)$ nemlineáris kétváltozós kifejezésfüggvényt,
%akkor $G=\aba$.
If $G$ is a closed group such that $\aba \leq G \leq \sym$ and $G$ preserves a non-linear binary term function $f(x,y)$ then $G=\aba$.
\end{lemma}

\begin{proof}
%Ha egy $\phi$ permutáció megőrzi az $xy=x \wedge y$ vagy az $xy+x+y=x \vee y$ műveletek
%valamelyikét, akkor $\phi$ automorfizmus. Ez azért igaz, mert minden hálót meghatároz önmagában
%egy is a $\wedge$ és a $\vee$ műveletek közül, és egy Boole-algebra háló ezekre a műveletekre.
It follows from Lemma \ref{metszet_eleg} and Corollary \ref{omkat2} that if a permutation $\phi$ preserves $xy=x\wedge y$ or $xy+x+y=x \vee y$ then $\phi$ must be an automorphism. 
%This holds because every term function of $\ba$ is first-order definable alone from $\wedge$ or $\vee$.

%This holds because every lattice is determined alone by
%$\wedge$ or $\vee$, and $\ba$ equipped with these two operations is a lattice.

\begin{itemize}
%\item Legyen $f(x,y)=xy+x$, ekkor $f(x,f(x,y))=x(xy+x)+x=xy$ miatt minden $f$-et megőrző permutáció automorfizmus,
%ugyanez igaz a $xy+y$ műveletre is.
%\item Legyen $g(x,y)=xy+x+1$, ekkor $g(g(x,y),y)=(xy+x+1)y+(xy+x+1)+1=xy+x+y$ miatt minden $g$-t
%megőrző permutáció automorfizmus,
%ugyanez igaz a $xy+y+1$ műveletre is.
%\item Legyen $h(x,y)=xy+1$, ekkor $h(h(x,x),h(y,y))=(xx+1)(yy+1)+1=xy+x+y$ miatt minden $h$-t megőrző permutáció automorfizmus.
%\item Végül legyen $k(x,y)=xy+x+y+1$, ekkor $k(k(x,x),k(y,y))=(x+1)(y+1)+(x+1)+(y+1)+1=xy$ miatt minden $k$-t megőrző permutáció is automorfizmus.
\item Let  $f(x,y)=xy+x$ then $f(x,f(x,y))=x(xy+x)+x=xy$. So every permutation which preserves $f$ must be an automorphism. The same can be
said for the operation $xy+y$.
\item Let  $g(x,y)=xy+x+1$ then $g(g(x,y),y)=(xy+x+1)y+(xy+x+1)+1=xy+x+y$. So every permutation which preserves $g$ must be an automorphism. The same can be
said for the operation $xy+y+1$.
\item Let $h(x,y)=xy+1$ then $h(h(x,x),h(y,y))=(xx+1)(yy+1)+1=xy+x+y$. So every permutation which preserves $h$ must be an automorphism.
\item Let $k(x,y)=xy+x+y+1$ then $k(k(x,x),k(y,y))=(x+1)(y+1)+(x+1)+(y+1)+1=xy$. So every permutation which preserves $k$ must be an automorphism.
\end{itemize}
We have checked all possible non-linear binary term functions, and thus the lemma is proved.
\end{proof}

%Ha $f$ egy tetszõleges Boole-függvény, melynek aritása $k$, akkor
%$$f(x_1,...,x_k)=\sum_{\eps\in {}^k2}{\ee_{\eps}\prod_{i=1}^k x_i^{\varepsilon_i}}$$
%alakban is felírható, ahol minden $\ee_{\eps}$ értéke 0 vagy 1.
%A továbbiakban egy $\eps$ $k$ hosszú $0-1$ vektorra $|\eps|$ jelöli $\sum_{i=1}^k{\varepsilon_i}$-t. 

If $f$ is an arbitrary term function of arity $k$ then $f$ can be written as
$$f(x_1,\dots,x_k)=\sum_{\eps\in \{0,1\}^k}{\ee_{\eps}\prod_{i=1}^k x_i^{\varepsilon_i}}$$
where all coefficients $\ee_{\eps}$ are $0$ or $1$. If $\eps$ is a $0-1$ vector of length $k$ then $\sum_{i=1}^k{\varepsilon_i}$ will be denoted by 
$|\eps|$. 

%Jelölje továbbá $f_{x_i x_j}(x_1,\ldots x_{k-1})=f(x_1, \ldots x_{j-1}, x_i, x_{j+1}, \ldots x_k)$ amikor $f$-ben $x_j$ helyébe $x_i$-t helyettesítünk.
%Amennyiben ez nem okoz félreértést, $f_{x_i x_j}$ helyett írhatunk $f_{ij}$-t. Szükségünk lesz az alábbi definícióra is:

We will denote the $(k-1)$-ary function $f(x_1, \ldots, x_{j-1}, x_i, x_{j+1}, \ldots, x_k)$ by $f_{x_i x_j}(x_1,\ldots, x_{k-1})$. By an abuse of
notation we will write $f_{ij}$ instead of $f_{x_i x_j}$ if the meaning is clear from the context. It is clear that from this definition that if a permutation $\phi$ preserves $f$, then it also preserves $f_{ij}$.

	We will need the following definition:

\begin{definition}
%A $\ba$ struktúra egy $c$ elemével való eltoláson a $t_c(a)=a+c$ permutációt értjük. Az összes eltolás csoportját $T$-vel
%fogjuk jelölni:
Let $c$ be an arbitrary element of $\ba$. Then the translation by $c$ is the permutation $t_c(a)=a+c$. The set of all translations will be denoted by $T$.
$$T=\left\{ t_c : c \in \ba\right\}$$
\end{definition}

%Most karakterizáljuk a háromváltozós nemlineáris kifejezésfüggvények lehetséges automorfizmus-csoportjait:

We will now characterize the possible automorphism groups corresponding to ternary non-linear term functions.

\begin{lemma}\label{eltolasok}
	Let $f$ be a ternary non-linear term function expressed as 
 $$f(x,y,z)=xy+yz+zx + \al_3 x + \al_2 y + \al_1 z + \al_0$$ where there are exactly zero or two $1$s among $\al_1 ,\al_2$ and $\al_3$. Then a permutation $\phi$ preserves $f$ if and only if $\phi$ can be written as $\phi= t \circ \psi$ where $t$ is a translation and $\psi$ is an automorphism of $\ba$.
\end{lemma}

\begin{proof}
Since $f$ is a term function over $\ba$, it follows that every automorphism of $\ba$ also preserves $f$. Now we show that $f$ is also preserves by every translation. Let $t=t_c$. Then we have
%If there are exactly zero or two $1$s are among $\al_1 ,\al_2$ and $\al_3 $ then for every element $c$ of $\ba$ the translation by $c$ ($t_c(a)=a+c$) preserves $f$:

\begin{multline*}
f(t(x),t(y),t(z))=f(x+c,y+c,z+c)=\\=xy +yz + zx +c +\al_3 x +\al_3 c+ \al_2 y + \al_2 c+ \al_1 z + \al_1 c+ \al_0 =\\= f(x,y,z) + (\al_1 + \al_2 +\al_3 +1)c=f(x,y,z)+c=t(f(x,y,z)).
\end{multline*}

	For the other direction let $\phi$ be an arbitrary permutation which preserves $f$. Denote the permutation $t_{\phi(0)} \circ \phi$ by $\thi$ (then $\thi (x)=\phi(x) + \phi(0)$). We will show that this permutation is an automorphism: $\thi \in \aba$. Since $\thi$ preserves $f$ and $\thi(0)=0$, the permutation $\thi$ must preserve
the function $g(x,y)=f(x,y,0)$. This $g$ is a binary non-linear term function, so $\thi$ is an automorphism by Lemma \ref{ketvalt}. The decomposition
of $\phi$ as $t_{\phi(0)}^{-1} \circ \thi = \phi$ is of the required form.
\end{proof}

\begin{lemma}\label{haromvalt}
%Legyen $f(x,y,z)=\al_7 xyz+ \al_6 xy +\al_5 yz + \al_4 zx + \al_3 x + \al_2 y + \al_1 z + \al_0$
%egy nemlineáris háromváltozós kifejezésfüggvény. Ekkor
%minden olyan $\phi$ permutációra, amely eleme $\aut{f}$-nek, a következő két lehetőség egyike teljesül:
Let $f$ be a ternary non-linear term function expressed as 
 $$f(x,y,z)=\al_7 xyz+ \al_6 xy +\al_5 yz + \al_4 zx + \al_3 x + \al_2 y + \al_1 z + \al_0$$
where  at least one of the coefficients $\al_7,\al_6,\al_5$ and $\al_4$ is equal to $1$. Let $\phi$ be a permutation which
preserves $f$. Then one of the following two possibilities holds for $\phi$:
\begin{itemize}
%\item $\phi$ eleme $\aba$-nak
%\item $\phi$ előáll egy $\aba$-beli elemnek és egy nem-identikus eltolásnak a kompozíciójaként.
\item $\phi$ is an automorphism: $\phi \in \aba$.
\item $\phi$ can be obtained as the composition of an automorphism and a non-identical translation.
\end{itemize}
\end{lemma}

\begin{proof}
%Ha az $f_{xy},f_{yz},f_{zx}$ függvények legalább egyike nemlineáris, akkor \aref{ketvalt}. Lemma miatt készen vagyunk.
Since the permutation $\phi$ preserves $f$, it also preserves the function $f_{xy},f_{yz}$ and $f_{zx}$. Therefore if at least one of the functions $f_{xy},f_{yz}$ and $f_{zx}$ are non-linear then it follows from Lemma \ref{ketvalt} that $f\in \aba$, and therefore the statement of the lemma holds.
\begin{itemize}
\item
%Ha $\al_7 =1$ és $\al_6=\al_5$ akkor $f_{yz}$ nemlineáris. Az $\al_5 = \al_4$ és a $\al_4 = \al_6$ alesetek hasonlóak,
%és ezek közül legalább az egyik fennáll. 
If $\al_7=1$ and $\al_6=\al_5$ then $f_{yz}$ is a binary non-linear term function. Since there are two identical values among $\al_4, \al_5$ and
$\al_6$ in the case of $\al_7=1$ the statement of the Lemma will hold.
%\item Ha $\al_7 =0$ és $\al_4, \al_5, \al_6$ mindegyike $1$, akkor bármely két változó azonosítása nemlineáris függvényt eredményez.
%HIBÁS, FELESLEGES SOR!!!
%\item Ha $\al_7 =0$ és $\al_4, \al_5, \al_6$ közül pontosan az egyik $1$, feltehetjük, hogy $\al_4$ az, akkor $f_{yz}$ megfelel.

\item If $\al_7 =0$ and exactly one of $\al_4, \al_5, \al_6$ is $1$: by symmetry we can assume that $\al_4=1$, then $f_{yz}$ is non-linear.

\item If $\al_7 =0$ and exactly one of $\al_4, \al_5, \al_6$ is $0$: by symmetry we can assume that $\al_4=0$, then $f_{xz}$ is non-linear.

\item If $\al_7 =0$ and $\al_4=\al_5=\al_6=1$ then $f$ can be written as
$$f(x,y,z)=xy +yz + zx +\al_3 x + \al_2 y + \al_1 z + \al_0.$$ The case when there are exactly zero or two $1$s are among $\al_1 ,\al_2$ and $\al_3$ is handled in Lemma \ref{eltolasok}. In the case when there are exactly one or three $1$s among $\al_1 ,\al_2  , \al_3 $ then
$f(a,a,a)=(f_{yz})_{xy}(a)=0$ or $f(a,a,a)=1$ for every $a \in \ba$. In this case $g(x,y)=f(x,y,f(x,x,x))$ is a binary non-linear term function, and then we can apply Lemma \ref{ketvalt} again.
\end{itemize}
\end{proof}

%Ha $\al_1 ,\al_2  , \al_3 $ között páratlan sok $1$ van, akkor $f(a,a,a)=(f_{yz})_{xy}(a)=0$ vagy $f(a,a,a)=1$ minden $a \in \ba$-ra, ezért
%a $g(x,y)=f(x,y,f(x,x,x))$ függvény egy nemlineáris kétváltozós kifejezésfüggvény, így minden az $f$-et, és így $g$-t is megőrző $\phi$ permutációra
%$\phi \in \aba$ \aref{ketvalt}. Lemma miatt.
%\end{proof}

\begin{comment}
%Tehát ha $f(x,y,z)=xy +yz + zx +\al_3 x + \al_2 y + \al_1 z + \al_0$ alakú, ahol $\al_1 ,\al_2  , \al_3 $ között az $1$-esek száma $0$ vagy $2$,
%akkor az alábbi két feltétel ekvivalens:
Let $f$ be of the form $f(x,y,z)=xy +yz + zx +\al_3 x + \al_2 y + \al_1 z + \al_0$ where exactly zero or two of the coefficients $\al_1 ,\al_2 $ and $ \al_3 $
are equal to $1$.
In the proof of Lemma \ref{haromvalt} we proved that the following two statements are equivalent for a permutation $\phi$:

\begin{itemize}
%\item Egy $\phi$ permutáció megőrzi az $f$ függvényt.
%\item Egy $\phi$ permutáció előáll $\phi= t \circ \psi$ alakban, ahol $t$ egy eltolás, és $\psi \in \aba$
\item $\phi$ preserves $f$
\item $\phi$ can be written as $\phi= t \circ \psi$ where $t$ is a translation and $\psi$ is an automorphism.
\end{itemize}
%Jelöljük $M$-mel a mediáns műveletét: $M(x,y,z)=xy+yz+zx$, továbbá
%jelöje $\aut{M}$ a mediáns műveletét megőrző permutációk csoportját.
\end{comment}

\begin{theorem}\label{msemidir}
Let $M$ denote the ternary function median: $M(x,y,z)=xy+yz+zx$. This is the same as the usual lattice-theoretic (lower or upper) median
$(x\vee y)\wedge(y \vee z)\wedge(z \vee x)$. Let $\Aut(M)$ denote the group of all permutations preserving $M$. Then $\Aut(M)=T \rtimes \aba$
\end{theorem}

\begin{proof}
%Mivel $\aut{M}$ minden eleme megőrzi $M(x,y,z)=xy+yz+xz$-et, így $\aut{M}$ minden $\phi$ eleme előáll
%$\phi= t \circ \psi$ alakban ($t\in T$ és $\psi \in \aba$) \aref{haromvalt}. Lemma alapján.
%Továbbá $\aba$ és $T$ minden eleme megőrzi a mediánst, így
%$\aut{M}=\left< \aba , T \right>$.

By Lemma \aref{eltolasok} every permutation of $\Aut(M)$ can be written as $\phi= t \circ \psi$ (where $t\in T$ and $\psi \in \aba$).
%because every permutation of $\Aut(M)$ preserves $M(x,y,z)=xy+yz+xz$
 Moreover, every permutation in $\aba$ and $T$ preserves the median so $\Aut(M)=\left< \aba , T \right>$.

The intersection of the groups $\aba$ and $T$ is the trivial group because every element of $\aba$ preserves  $0$ and the identity is the only $0$-preserving
translation.

Finally, the group $T$ is a normal subgroup of $\Aut(M)=\left< \aba , T \right>$. For this it is enough to show that the group $T$ is closed under conjugation of elements of $\aba$. Let $t_c$ be an arbitrary element of $T$ and $\phi$ be an arbitrary element of $\aba$. Then

$$(\phi^{-1} \circ t_c \circ \phi)(x)=\phi^{-1}(\phi(x)+c)=(\phi^{-1} \circ \phi)(x) + \phi^{-1}(c)= t_{\phi^{-1}(c)}(x)$$
so   $\phi^{-1}\circ t_c \circ \phi \in T$. 

\end{proof}

%A háromnál több változós nemlineáris kifejezésfüggvények esetét vissza fogjuk vezetni a legfeljebb három változós esetre.

The case of the non-linear term functions with more than three variables can be reduced to the previous cases by the following lemma.

\begin{lemma}\label{ind}
%Legyen $f$ egy nemlineáris Boole-függvény, amelynek aritása $k$, és $k$ legalább $4$. Ekkor léteznek olyan
%$1 \leq i < j \leq k$
%indexek melyekre $f_{ij}$ is nemlineáris.

Let $f$ be a non-linear term function of arity $k\geq 4$. Then there exist indices $1\leq i < j \leq k$ such that the function $g=f_{ij}$ is non-linear.
\end{lemma}

\begin{proof}
%Let 
%$$f(x_1,\dots,x_k)=\sum_{\eps\in \{0,1\}^k}{\ee_{\eps}\prod_{i=1}^k x_i^{\varepsilon_i}}$$
% and let $k$ denote the arity of $f$. To prove  the statement of the lemma we  show that
%if $k$ is at least $4$, then there are two indices $i$ and $j$ such that
%$1 \leq i < j \leq k$ and the function $g=f_{ij}$ is non-linear. 
We will prove the statement by checking the following three cases:
\begin{itemize}
\item
The first case is when there is an $\eps$ such that $2 \leq |\eps | \leq k-2$ and $\ee_{\eps}=1$ holds. We have two indices $i$ and $j$
such that $1 \leq i < j \leq k$ and $\eps_i=\eps_j=0$. Then for this $\eps$ the term 
$\prod_{i=1}^k x_i^{\varepsilon_i}$ is with coefficient 1 in the functions
$f$ and  $f_{ij}$, as it is not altered by collapsing
$x_i$ and $x_j$ and this term cannot show up by collapsing those two variables, either.
So $f_{ij}$ is non-linear.

\item 
%Ha az egyetlen legalább másodfokú tag az összes változó szorzata,
%akkor bármelyik $i,j$ pár megfelelő.
If the only non-linear monomial is the product of all the variables then every pair of indices $i$ and $j$ is suitable.

\item

The remaining case is when all non-linear monomials have a degree at least $(k-1)$ and there is a monomial with a degree of exactly $(k-1)$. Let $\eps$ denote a
vector corresponding to such monomial. Let $i$ and $j$ denote two indices such that $\eps_i=\eps_j=1$. Then if we replace $x_j$ with $x_i$ the monomial
of $f$ corresponding to $\eps$ will become a monomial of $f_{ij}$ with a degree of $(k-2)$. This monomial of $f_{ij}$ cannot be formed from any other monomial of
$f$,  so it will not be cancelled out.
\end{itemize}
\end{proof}

\begin{lemma}\label{nemlinear1}
	Let $f$ be a non-linear term function. Then $\aba \leq \Aut(f) \leq \Aut(M)$.
\end{lemma}

\begin{proof}
	First we show that there exists a non-linear term function $g$ such that the arity of $g$ is at most $3$, and $\aba \leq \Aut(f) \leq \Aut(g)$. Let $k$ be the arity of $f$. If $k\leq 3$, then there is nothing to prove. If $k\geq 4$, then by Lemma \ref{ind} we know that $f_{ij}$ is non-linear for some indices $1\leq i<j\leq k$. Then $\aba\leq \Aut(f)\leq \Aut(f_{ij})$. By iterating this construction we can find a non-linear function $g$ of arity at most 3 such that $\aba \leq \Aut(f) \leq \Aut(g)$.

	By Lemma \ref{haromvalt} we have $\Aut(g)\subset \left< \aba , T \right>$. The right handside is equal to $\Aut(M)$ by Lemma \ref{msemidir}. Therefore $\aba\leq \Aut(f)\leq \Aut(g)\leq \Aut(M)$.
\end{proof}

%By Lemmas \ref{ketvalt}, \ref{haromvalt}, \ref{ind}, and \ref{msemidir} we can conclude that if $f$ is a non-linear term function, then $\aba \leq \Aut(f) \leq \Aut(M)$ holds. We will show that in fact $\aba  = \Aut(f)$ or $\Aut(f) = \Aut(M)$ must hold.

\begin{lemma}\label{nemlinear}
%Legyen $f$ nemlineáris kifejezésfüggvény.
%Ekkor $\aut{f}=\aba$ vagy $\aut{f}=\aut{M}$.
Let $f$ be an arbitrary non-linear term function. Then $\Aut(f)=\aba$ or $\Aut(f)=\Aut(M)$.
\end{lemma}

\begin{proof}

Let $\Aut(f)$ be denoted by $G$. The group $G$ is uniquely determined by the subgroup $T \cap G$ because $\Aut(M)=T \rtimes \aba$ and $\aba \leq G \leq \Aut(M)$.
Let $c$ and $d$ be two elements of $\ba$ such that they are on the same orbit of $\aba$. We will show that $t_c \in G$ implies $t_d \in G$.
Let $\phi \in \aba$ an automorphism such that $\phi(d)=c$ then
$$(\phi^{-1} \circ t_c \circ \phi)(x)=\phi^{-1}(\phi(x)+c)=(\phi^{-1} \circ \phi)(x) + \phi^{-1}(c)= t_{d}(x)$$
The automorphism group $\aba$ has three orbits on $\ba$: $\{0\},\{1\}$ and $\ba \setminus \{0,1\}$ so one of the following four possibilities must hold
(using that $t_0=\operatorname{id} \in T \cap G$):
\begin{itemize}
\item $T \cap G = \{t_c | c \in \{ 0 \}\}$
\item $T \cap G = \{t_c | c \in \{ 0, 1 \}\}$
\item $T \cap G = \{t_c | c \in \ba \setminus \{ 1 \} \}$
\item $T \cap G = \{t_c | c \in \ba \}$
\end{itemize}
%Ezek közül az első a $G=\aba$, a negyedik a $G=\aut{M}$ esetnek felel meg, a második és a harmadik lehetőségről
%pedig megmutatjuk, hogy nem állhatnak fenn.
The first case corresponds to the case $G=\aba$ and the fourth corresponds to the case $G=\Aut(M)$. We will rule out the second and third possibility.

%A harmadik lehetőséget kizárja, hogy a $T \cap G = \{t_c | c \in \ba \setminus \{ 1 \} \}$ halmaz nem is
%részcsoport: legyen $a \in \ba \setminus \{ 0, 1 \}$, ekkor $t_{a}$ és $t_{a+1}$ is eleme $T \cap G$-nek, így a
%kompozíciójuk $(t_{a+1} \circ t_a)(x)=x+a+a+1=t_1(x)$ is eleme kellene legyen.
The third possibility can be ruled out by noticing that the set $ \{t_c | c \in \ba \setminus \{ 1 \} \}$ is not a subgroup. Let
$a \in \ba \setminus \{ 0, 1 \}$ be an arbitrary element then the translations $t_{a}$ and $t_{a+1}$ must be in $ \{t_c | c \in \ba \setminus \{ 1 \} \}=T \cap G$
so their composition $(t_{a+1} \circ t_a)(x)=x+a+a+1=t_1(x)$ must also be contained in $T \cap G$.

%A második lehetőség megad egy létező $\aba \leq G \leq \aut{M}$ zárt csoportot, erről kell belátnunk, hogy nem áll elő
%funkcionális redukt automorfizmus-csoportjaként. Jelölje $\mathfrak{B}_2$ a kételemű Boole-algebrát. Legyen $f$ olyan
%kifejezésfüggvény, melyet megőriz a $t_1$ permutáció. Ekkor a
%$f(x_1 +1, x_2 +1 \ldots x_k +1)= f(x_1, x_2 \ldots x_k)+1$ és a
%$f(x_1 +0, x_2 +0 \ldots x_k +0)= f(x_1, x_2 \ldots x_k)+0$ azonosságok is teljesülnek $f$-re, azaz az
%$f(x_1 +c, x_2 +c \ldots x_k +c)= f(x_1, x_2 \ldots x_k)+c$ azonosság teljesül minden $c \in \mathfrak{B}_2$-re.
%Tehát ez az azonosság teljesül a $\mathfrak{B}_2$ által generált varietás minden algebrájában, így $\ba$-ban is.
%Így ha $G=\aut{f}$ valamilyen $f$ nemlineáris kifejezésfüggvényre, akkor
%$t_1 \in \aut{f}$-ből következik, hogy $t_c \in \aut{f}$ tetszőleges $c$-re.

The second case gives us a
%n existing 
closed group $\aba \leq G \leq \Aut(M)$. We will show that this group cannot be obtained as the automorphism group
of a functional reduct. 

	%Let $\mathfrak{B}_2$ denote the two-element Boolean algebra. 
	Let $f$ be a term function which is preserved by the translation $t_1$.
Then the following two equations hold:
$$f(x_1 +1, x_2 +1, \ldots, x_k +1)= f(x_1, x_2, \ldots, x_k)+1,$$
and clearly
$$f(x_1 +0, x_2 +0, \ldots, x_k +0)= f(x_1, x_2, \ldots, x_k)+0.$$
In particular for all $x_1,x_2,\ldots,x_k,y\in \{0,1\}$ we have 
$$f(x_1 +x, x_2 +x, \ldots, x_k +y)= f(x_1, x_2, \ldots, x_k)+y.$$
%As $\{0,1\}$ is a subalgebra isomorphic to $\mathfrak{B}_2$, $f$ preserves $\{0,1\}$, hence $f$ is a term function of $\mathfrak{B}_2$. 
	This means that the identity $$f(x_1 +x, x_2 +x, \ldots, x_k +y)= f(x_1, x_2, \ldots, x_k)+y$$ is satisfied in the subalgebra $\mathfrak{B}_2$. Therefore by Lemma \ref{rest} this identity also holds in $\ba$. Setting $y=c\in \ba$ we obtain that 
$$f(x_1 +c, x_2 +c, \ldots, x_k +c)= f(x_1, x_2, \ldots, x_k)+c$$ holds for all $x_1,\dots,x_k\in \ba$. Therefore $t_c$ preserves $f$ for every $c\in \ba$.

	Assume $G=\Aut(f)$ for some term function $f$. Then the argument above shows that $t_1\in \Aut(f)$ implies that $t_c \in \Aut(f)$ for arbitrary $c\in \ba$. This excludes the second case.
\end{proof}

%\Aref{omkat}. Következmény alapján így egy $f$ nemlineáris kifejezésfüggvényre az alábbi két eset pontosan egyike teljesül:
%By Corollary \ref{omkat} for a given non-linear term function $f$ exactly one of the following two possibilities holds:
%\begin{itemize}
%\item Az $f$ művelet és a szorzás művelete kölcsönösen definiálható egymásból.
%\item Az $f$ művelet és a mediáns művelete kölcsönösen definiálható egymásból.
%\item The function $f$ and the multiplication are first-order interdefinable.
%\item The function $f$ and the median are first-order interdefinable.
%\end{itemize}
%Ezzel befejeztük a nemlineáris kifejezésfüggvényeket megőrző permutációk leírását.

	Lemma \ref{nemlinear} can be formulated in terms of first-order definability as follows.

\begin{lemma}\label{nemlinear_fo}
	Let $f$ be an arbitrary non-linear term function. Then the function $f$ is first-order interdefinable with either the operation $\wedge$ or the median operation $M$.  
\end{lemma}

\begin{proof}
	By Lemma \ref{nemlinear} we know that either $\Aut(f)=\aba$ or $\Aut(f)=\Aut(M)$. By Lemma \ref{metszet_eleg} we have $\aba=\Aut(\wedge)$. Then the statement of the lemma follows directly from Corollary \ref{omkat}.
\end{proof}

\section{The linear functional reducts}

%Ebben a fejezetben a lineáris funkcionális reduktokat fogjuk meghatározni, kölcsönös
%definiálhatóság erejéig.
In this section we will classify the linear functional reducts up to first-order interdefinability.

%Minden lineáris kifejezésfüggvény $l(x_1, x_2 \ldots x_k)=x_1 + x_2 + \ldots + x_k + \al$ alakú, ahol $\al =0$
%vagy $\al=1$.
All linear term functions can be written as $l(x_1, x_2, \ldots, x_k)=x_1 + x_2 + \ldots + x_k + \al$ where $\al$ is either $0$ or $1$.

%Tekintsük az alábbi nyolc kifejezésfüggvényt:
Consider the following eight term functions:
\begin{enumerate}
\item $0$
\item $1$
\item $x$
\item $\neg(x)=x+1$
\item $+_0(x,y)=x+y$
\item $+_1(x,y)=x+y+1$
\item $\Sigma(x,y,z)=x+y+z$
\item $\Sigma_1(x,y,z)=x+y+z+1$
\end{enumerate}
%Ezekre a lineáris függvényekre mostantól mint \textbf{kanonikus lineáris függvényekre} fogunk hivatkozni, annak ellenére,
%hogy a konstans $1$ függvény nem lineáris: $f(a+b)=1 \neq 0=f(a)+f(b)$. A továbbiakban szükségünk lesz ezeknek a
%függvényeknek az automorfizmus-csoportjaira, most ezek leírása következik.
We will refer to this functions as \textbf{canonical linear functions}. We will need the corresponding automorphism groups:

\begin{csop}{$\Aut(0)$}
%Az $f=0$ esetben $\aut{0}$ a $0$ stabilizátora a $\sym$ csoportban. Mivel $\aut{0}$-hoz tetszőleges
%$\phi \notin \aut{0}$ a $0$-t nem fixáló permutációt generátorként hozzávéve már az egész $\sym$-et kapjuk, így
%ez a csoport maximális valódi részcsoportja $\sym$-nek.
$\Aut(0)$ is the stabilizer of $0$ in the whole symmetric group $\sym$. This is a maximal proper subgroup of $\sym$ because for every permutation
$\phi \notin \Aut(0)$ the subgroup $\left< \Aut(0), \phi \right>$ is the whole $\sym$.
\end{csop}

\begin{csop}{$\Aut(1)$}
%Az $f=1$ esetben $\aut{1}$ az $1$ stabilizátora a $\sym$ csoportban. Az $\aut{0}$ csoporthoz hasonlóan ez a csoport is
%maximális valódi részcsoportja $\sym$-nek.
Similarly $\Aut(1)$ is the stabilizer of $1$ in the whole symmetric group $\sym$. This is also a maximal proper subgroup.
\end{csop}

\begin{csop}{$\Aut(\emptyset)$}
 %Az $f(x)=x$ esetben $\aut{x}$ maga a $\sym$ csoport, amely természetesen tartalmazza az összes redukt automorfizmus-csoportját.
The group $\Aut(\emptyset)=\Aut(x)$ is the whole symmetric group.
\end{csop}
\begin{csop}{$\Aut(\neg)$}

For $\neg(x)=x+1$ let $\mxi$ denote an arbitrary maximal ideal of $\ba$. Define the following two groups:
Let $\simik$ be defined as an extension of the action of  $\simi$  to the whole $\ba$. For $\phi \in \simi$ define the action
 as $\phi (x)=\phi (x)$ if $x \in \mxi$ and $\phi(x)=\phi (x+1)+1$
if $x \notin \mxi$.

Let $Z_2^{\mxi}$ denote the group consisting of the following permutations: $\phi \in Z_2^{\mxi}$ if and only if for every $x \in \ba$
$\phi(x)=x$ or $\phi(x)=x+1$. This $Z_2^{\mxi}$ group is always the same regardless of the choice of the maximal ideal $\mxi$ because we do not refer to
$\mxi$ in its definition. The notation is justified because $Z_2^{\mxi}$ is isomorphic to a direct power of $Z_2$ where the direct factors are indexed with
the elements of $\mxi$.

%Megmutatjuk, hogy $\aut{\neg}=Z_2^{\mxi} \rtimes \simik$.
We will show that $\Aut(\neg)=Z_2^{\mxi} \rtimes \simik$.

The group $Z_2^{\mxi}$ is a normal subgroup in $\Aut(\neg)$. Let $\phi \in Z_2^{\mxi}$ and $\psi \in \Aut(\neg)$ be two permutations, we need that the
permutation $\psi^{-1} \phi \psi$ is also in $Z_2^{\mxi}$. This holds if and only if for every $x \in \ba$ the image $\psi^{-1} \phi \psi(x)$ is 
either $x$ or $x+1$. So we have the following two cases:
\begin{itemize}
%\item Ha $\phi(\psi(x))=\psi(x)$, akkor $\psi^{-1} \phi \psi(x)=x$.
%\item Ha $\phi(\psi(x))=\psi(x)+1$, akkor $\psi^{-1} \phi \psi(x)=\psi^{-1}(\neg\psi(x))=\neg\psi^{-1}(\psi(x))=x+1$.
\item If $\phi(\psi(x))=\psi(x)$ then $\psi^{-1} \phi \psi(x)=x$.
\item If $\phi(\psi(x))=\psi(x)+1$ then $\psi^{-1} \phi \psi(x)=\psi^{-1}(\neg\psi(x))=\neg\psi^{-1}(\psi(x))=x+1$.
\end{itemize}
The intersection of the groups $Z_2^{\mxi}$ and $\simik$ is trivial because there are no permutations in $Z_2^{\mxi}$ such that the image of any
element of $\mxi$ is another element of $\mxi$.

We will show that every $\phi$ permutation in $\Aut(\neg)$ can be written as the composition of a permutation from $\simik$ and a permutation from
$Z_2^{\mxi}$. Let $\phi$ denote an arbitrary permutation from $\Aut(\neg)$, first we will define a permutation $\thi$:
\begin{itemize}
%\item Ha $x \in \mxi$, akkor $\thi(x)$ legyen $\phi(x)$ és $\phi(x)+1$ közül az, amelyik $\mxi$-be esik.
%\item Ha $x \notin \mxi$, akkor $\thi(x)$ legyen $\phi(x)$ és $\phi(x)+1$ közül az, amelyik nem esik $\mxi$-be.
\item If $x \in \mxi$ then let $\thi(x)$ be the element from $\{\phi(x),\phi(x)+1\}$ the one which is in $\mxi$.
\item If $x \notin \mxi$ then let $\thi(x)$ be from $\{\phi(x),\phi(x)+1\}$ the one which is not in $\mxi$.
\end{itemize}

This $\thi$ will be an element of $\simik$ and $\thi^{-1} \circ \phi$ will be an element of $Z_2^{\mxi}$. The composition of these two permutations is $\phi$.
\end{csop}

\begin{csop}{$\Aut(+_0)$}\label{vter}

In case of the operation $+_0(x,y)=x+y$ the functional reduct is a vector space $\mathbb{F}_2^{\infty}$.
It has the automorphism group $\Aut(+_0)=\on{GL}(\infty,2)$.
\end{csop}
\begin{csop}{$\Aut(+_1)$}
%A $+_1(x,y)=x+y+1$ esetben a redukt szintén egy vektortér, melynek műveletei nem azonosak a $+_0(x,y)=x+y$
%vektortér műveleteivel (például $+_0$ nulleleme a $0$, $+_1$ nulleleme az $1$).
%Viszont kettejük között megadható egy izomorfizmus: a $\tau_1$ eltolás, ugyanis
%$\tau_1(+_1(x,y))=x+y+1+1=x+1+y+1=+_0(\tau_1(x),\tau_1(y))$ és
%$\tau_1^{-1}(+_0(x,y))=x+y+1=+_1(\tau_1^{-1}(x),\tau_1^{-1}(y))$.
%Ennek alapján vezessük be a következő jelölést: $\aut{+_1}=\on{GL}^1(\infty,2)$.
%Izomorf struktúrák automorfizmus-csoportja is izomorf, azaz
%$\aut{+_1}=\on{GL}^1(\infty,2) \cong \on{GL}(\infty,2)=\aut{+_0}$.
In case of the operation $+_1(x,y)=x+y+1$ the functional reduct is also a vector space, but in this case the zero element of the vector space is 
$1$, and the addition is the operation $+_1(x,y)=x+y+1$. This vector space is isomorphic to the vector space in Group~\ref{vter}, the $\tau_1$ translation is an isomorphism:
$$\tau_1(+_1(x,y))=x+y+1+1=x+1+y+1=+_0(\tau_1(x),\tau_1(y))$$
and
$$\tau_1^{-1}(+_0(x,y))=x+y+1=+_1(\tau_1^{-1}(x),\tau_1^{-1}(y))$$
We will denote the automorphism group $\Aut(+_1)$ by $\on{GL}^1(\infty,2)$.
Since the two vector spaces are isomorphic, their automorphism groups are also isomorphic: $\Aut(+_1)=\on{GL}^1(\infty,2) \cong \on{GL}(\infty,2)=\Aut(+_0)$.
\end{csop}
\begin{csop}{$\Aut(\Sigma)$}
%A $\Sigma(x,y,z)=x+y+z$ esetben a funkcionális redukt egy affin tér. Legyen $x,y,z,v \in \ba$ négy páronként különböző elem.
%Ezekre $\Sigma(x,y,z)=v$ akkor és csak akkor áll fenn, ha az $x,y,z,v$ elemek egy kétdimenziós affin alteret alkotnak:
%$\Sigma(x,y,z)=v \Leftrightarrow x+y+z=v$, tehát az $\{ x,y,z,v\} $ elemek a $\{0,y+x,z+x,v+x\}$
%kétdimenziós lineáris altér $x$-szel való eltoltját alkotják.
In  case of the operation $\Sigma(x,y,z)=x+y+z$ the functional reduct is an affine space. Let $x,y,z$ and $v$ be four pairwise different elements
from $\ba$. Then $\Sigma(x,y,z)=v$ holds if and only if these four elements form a two dimensional affine subspace. If 
$\Sigma(x,y,z)=v \Leftrightarrow x+y+z=v$ then $\{ x,y,z,v\} $ is a translate of the linear subspace $\{0,y+x,z+x,v+x\}$ by $x$.

%$\Sigma$-t megőrzik (az $x_0$ műveletet is megőrző) $\on{GL}(\infty,2)$-beli permutációk és az eltolások is.
%Továbbá a $T$ eltolások csoportja normálosztó $\aut{\Sigma}$-ban, $T$ és $\on{GL}(\infty,2)$ együtt generálják az egész
%$\aut{\Sigma}$-t, és metszetük csak az identikus permutáció.
%Azaz az $\aut{\Sigma}$ automorfizmus-csoport előáll mint a  $T \rtimes \aut{+_0}$ szemidirekt szorzat.
%Szimmetriaokokból $\aut{\Sigma}$ mint a $T \rtimes \aut{+_1}$ szemidirekt szorzat is felírható.

The operation $\Sigma$ is preserved by all translations $T$ and all linear transformations $\on{GL}(\infty,2)$. The group of translations $T$ is a normal
subgroup in $\Aut(\Sigma)$ and $\Aut(\Sigma)$ is generated by the subgroups $T$ and $\on{GL}(\infty,2)$. So $\Aut(\Sigma)$ is a semidirect product
$T \rtimes \Aut(+_0)$.
\end{csop}
\begin{csop}{$\Aut(\Sigma_1)$}
%Végül a $\Sigma_1(x,y,z)=x+y+z+1$ esetben $\Sigma_1$ segítségével definiálhatjuk a $\neg(x)=g(x,x,x)$ és a
%$\Sigma(x,y,z)=\Sigma_1(x,y,\neg(z))$ műveleteket. Másrészt $\neg(x)$ és $\Sigma(x,y,z)$ segítségével is definiálhatjuk
%$\Sigma_1(x,y,z)$-t a következő módon: $\Sigma_1(x,y,z)=\Sigma(x,y,\neg(z))$.
%Tehát \aref{omkat2}. Következmény alapján $\aut{\Sigma_1}=\aut{\Sigma} \cap \aut{\neg}$.
In  case of  operation $\Sigma_1(x,y,z)=x+y+z+1$ we can define both $\neg(x)=\Sigma(x,x,x)$ and $\Sigma(x,y,z)=\Sigma_1(x,y,\neg(z))$.
We can also define $\Sigma_1(x,y,z)$ using $\neg(x)$ and $\Sigma(x,y,z)$: $\Sigma_1(x,y,z)=\Sigma(x,y,\neg(z))$.
By Corollary \ref{omkat2} $\Aut(\Sigma_1)=\Aut(\Sigma) \cap \Aut(\neg)$.

%Mivel az eltolások megőrzik $\Sigma_1$-et, így $T$ részcsoportja $\aut{\Sigma_1}$-nek.
%Felhasználva, hogy $T \triangleleft \aut{\Sigma}$, a $T$ csoport normálosztó lesz $\aut{\Sigma_1}$-ban is.
$T$ is a subgroup of $\Aut(\Sigma_1)$ because all translations preserve $\Sigma_1$. Using the fact that $T \triangleleft \Aut(\Sigma)$
the group $T$ will be a normal subgroup in $\Aut(\Sigma_1)$, too.

%Az $\aut{+_0}=\on{GL}(\infty,2)$ csoportnak az $\aut{\Sigma_1}$ csoportba eső része a komplementert tartó lineáris
%permutációkból áll. Belátjuk, hogy ezek pontosan az $1$-et fixáló lineáris permutációk: legyen
%$\phi \in \aut{+_0}$ komplementertartó. Mivel $\phi$ fixálja a $0$-t, így $\phi$ fixálja $\neg 0=1$-et is.
%Fordítva, ha $\phi \in \aut{+_0}$ megőrzi az $1$-et, akkor megőrzi a komplementerséget is:
%$\phi(x+1)=\phi(x)+\phi(1)=\phi(x)+1$.
%Az $1$-et fixáló lineáris permutációk csoportja megegyezik a $\aut{+_0}\cap \aut{+_1}$ csoporttal.

The group $\on{GL}(\infty,2)  \cap \Aut(\Sigma_1)$ consists of the complement-preserving linear permutations. These are exactly
those linear permutations which fix the $1$. We can conclude that

%Az előzőekből következik, hogy
$$\Aut(\Sigma_1)=\Aut(\Sigma) \cap \Aut(\neg)=(T \rtimes \Aut(+_0))\cap \Aut(\neg)=$$
$$=T \rtimes (\Aut(+_0) \cap \Aut(\neg))=T \rtimes (\Aut(+_0)\cap \Aut(+_1))$$
\end{csop}
\par\addvspace{\baselineskip}
\begin{lemma}\label{linear}
%Minden $l$ lineáris kifejezésfüggvényhez létezik pontosan egy $f$ kanonikus lineáris függvény, amelyre
%teljesül, hogy $l$ és $f$ kölcsönösen definiálhatóak egymásból elsőrendű formulák segítségével.
For every linear term function $l$, there uniquely exists a canonical linear function $f$ such that $l$ and $f$ are first-order interdefinable.
\end{lemma}
\begin{proof}
%Legyen $l(x_1, x_2 \ldots x_k)=x_1 + x_2 + \ldots + x_k + \al$.

%Először belátjuk olyan $f$ kanonikus lineáris függvény létezését, amelyre $f$ és $l$ kölcsönösen definiálható egymásból,
%az egyértelműség bizonyítását később végezzük el.

%Minden olyan függvény, amelyre $k<2$, szerepel a listában, így választhatjuk önmagukat a megfelelő $f$-nek.

Let $l$ denote $l(x_1, x_2, \ldots, x_k)=x_1 + x_2 + \ldots + x_k + \al$.
First we will prove the existence of such $f$, then prove the uniqueness.
If the arity of $l$ is at most one then $l$ is a canonical linear function so we can choose itself as $f$.

Let $k$ denote the arity of $l$. If $k$ is even and at least two, then $f(x,y)=x+y+\al$ will be first-order interdefinable with $l$.
The function $f$ can be defined from $l$ because $f(x,y)=l(x,y,y, \ldots, y)$. For the other direction define the following series of functions: let $l_2$ be
$l_2=f(x_1,x_2)$ and recursively $l_j(x_1, x_2, \ldots, l_j)=f(l_{j-1}(x_1,x_2, \ldots, x_{j-1}),x_j)$ if $j>2$. Then $l=l_k$ so $l$ can be
defined from $f$.

%Ha $k>1$ páros szám akkor az $f(x,y)=x+y+1$ jó választás. Az egyik irányú definiálhatóságot
%$f(x,y)=l(x,y,y \ldots y)$ bizonyítja, a másik irányhoz definiáljuk az alábbi függvényeket rekurzívan:
%$l_2=f(x_1, x_2)$, $l_j(x_1, x_2 \ldots l_j)=f(l(x_1,x_2 \ldots x_{j-1}),x_j)$. Ekkor $l=l_k$.

%Ha pedig $k>1$ páratlan szám akkor az $f(x,y,z)=x+y+z+\al$ jó választás.
%Az egyik irányú definiálhatóságot
%$f(x,y,z)=l(x,y,z,z \ldots z)$ bizonyítja, a másik irányhoz definiáljuk az alábbi függvényeket rekurzívan:
%$l_1=f(x_1,x_2,x_3)$ és $j>3$-ra
%$$l_j(x_1, x_2 \ldots x_{2j+1})=f(l(x_1,x_2 \ldots x_{2j-1}),f(x_{2j},x_{2j},x_{2j}),f(x_{2j+1},x_{2j+1},x_{2j+1}))$$.
%Ekkor $l=l_{(k-1)/2}$ .

If $k$ is odd and at least $3$ then $f(x,y,z)=x+y+z+\al$ will be first-order interdefinable with $l$. The function $f$ can be defined from $l$
because $f(x,y,z)=l(x,y,z,z, \ldots, z)$. Define the following series of functions: let $l_1$ be  $l_1=f(x_1,x_2,x_3)$ and recursively
$$l_j(x_1, x_2, \ldots, x_{2j+1})=$$
$$=f(l_{j-1}(x_1,x_2, \ldots, x_{2j-1}),f(x_{2j},x_{2j},x_{2j}),f(x_{2j+1},x_{2j+1},x_{2j+1})).$$
Then $l=l_{(k-1)/2}$ so $l$ can be defined from $f$.

%Be kell látnunk még a megfelelő $f$ egyértelműségét. Ha lenne olyan $l$ lineáris kifejezésfüggvény, amihez
%több olyan $f$ kanonikus lineáris függvény is létezne, hogy $l$ és $f$ kölcsönösen
%definiálhatóak egymásból elsőrendű formulák segítségével,
%akkor ezek az $f$-ek is kölcsönösen definiálhatóak lennének egymásból, így megegyezne az automorfizmus-csoportjuk is.
%Ezért \aref{omkat}. Következmény alapján $f$ egyértelműségét bizonyíthatjuk úgy,
%ha a listánk minden függvényének meghatározzuk az
%automorfizmus-csoportját, és ezek különbözőnek bizonyulnak.

We will show the uniqueness of the previous $f$. If for a given $l$ there are two different canonical linear functions
$f_1,f_2$ such that $l$ is first-order interdefinable with both, then $f_1$ will be interdefinable with $f_2$.
So their automorphism group will be the same. Using Corollary \ref{omkat} it is enough to show that the automorphism
groups of two different canonical linear functions are different.
We have described these automorphism groups, and they are pairwise distinct.
\end{proof}

\nc{\s}{\cdot}
\begin{figure}\label{frhalo}
\centering
\begin{tikzpicture}[scale=2]
\tikzstyle{every node}=[shape=rectangle, draw, fill=white!]
\tikzstyle{white}=[shape=rectangle, draw, fill=white!]
\tikzstyle{green}=[shape=rectangle, draw, fill=white!, font=\large]
\tikzstyle{blue}=[shape=rectangle, draw, fill=white!]
\tikzstyle{red}=[shape=rectangle, draw, fill=white!]
\tikzstyle{ures}=[shape=circle, draw, fill=white!]

\tikzstyle{k}=[draw=none, fill=none]

\path

(1, 2)  node[k] (100)  {}
(0, 1)  node[k] (101)  {}
(2, 1)  node[k] (102)  {}
(1, 0)  node[k] (103)  {}
(3, 1)  node[k] (104)  {}
(1, -1)  node[k] (105)  {}
(0, -1)  node[k] (106)  {}
(1, -2)  node[k] (107)  {}
(1, -3)  node[k] (108)  {}
(2 ,-1)  node[k] (109)  {}
(1, 1)  node[k] (110)  {}
(3, -1)  node[k] (111)  {}
(3, -2)  node[k] (112)  {}
;
\draw
(101) -- (100)
(102) -- (100)
(103) -- (101)
(103) -- (102)
(104) -- (100)
(111) -- (104)
(105) -- (103)
(105) -- (104)
(108) -- (107)
(106) -- (110)
(109) -- (110)
(107) -- (109)
(110) -- (100)
(112) -- (108)
(111) -- (110)
(109) -- (102)
(106) -- (101)
(107) -- (106)
(107) -- (105)
(111) -- (107)
(112) -- (111)
;
\path
(1, 2)  node[blue] (0)  {{$\emptyset$}}
(0, 1)  node[blue] (1)  {{$0$}}
(2, 1)  node[blue] (2)  {{$1$}}
(1, 0)  node[blue] (3)  {{$0,1$}}
(3, 1)  node[blue] (4)  {{$\neg$}}
(1, -1)  node[blue] (5)  {{$\neg , 0, 1$}}
(0, -1)  node[blue] (6)  {{$+_0,0$}}
(1, -2)  node[blue] (7)  {{$+_0,0,1$}}
(1, -3)  node[blue] (8)  {{$\s ,+_0,0,1$}}
(2 ,-1)  node[blue] (9)  {{$ +_1,1$}}
(1, 1)  node[blue] (10)  {{$\Sigma$}}
(3, -1)  node[blue] (11)  {{$\Sigma_1$}}
(3, -2)  node[blue] (12)  {{$M$}}

;

%%rendezes a file elejere lett szamuzve

\end{tikzpicture}
\caption{\label{figure:func}The lattice of the functional reducts, ordered by the inclusion of their automorphism groups.}
\end{figure}

Our goal is to show that the automorphism groups of the functional reducts of $\ba$ ordered by inclusion form the lattice on
Figure \ref{figure:func}.

First we will prove that the (supposed) coatoms of the lattice form an antichain. Next we will describe the possible intersections
of the coatoms. Finally we determine the places of the two possible automorphism groups of the non-linear functional reducts.
There are three groups which can be obtained as the intersection of some coatoms and have not been characterized yet:

\begin{csop}{$\Aut(0,1)$}
%Az $\aut{0,1}$ csoport a $0$ és $1$ elemek pontonkénti stabilizátora. Mivel tetszőleges $\phi$ az $1$-et nem fixáló
%permutációt generátorelemként hozzávéve már a teljes $\aut{0}$-t kapjuk, így $\aut{0}$-nak maximális valódi részcsoportja.
%Hasonlóan maximális valódi részcsoportja $\aut{1}$-nek is.
The group $\Aut(0,1)$ is the pointwise stabilizer of $\{0,1\}$. It is a proper maximal subgroup of $\Aut(0)$ and $\Aut(1)$.
\end{csop}
\begin{csop}{$\Aut(\neg,0,1)$}

The group $\Aut(\neg,0,1)$ is the intersection of the groups $\Aut(0),\Aut(1)$ and $\Aut(\neg)$. Like the group $\Aut(\neg)$
can be obtained as a semidirect product $Z_2^{\mxi} \rtimes \simik$ the group $\Aut(\neg,0,1)$ can be decomposed as
$\Aut(\neg,0,1)=\left( Z_2^{\mxi} \right)_0 \rtimes \left( \simik \right)_0$. Here $\left( Z_2^{\mxi} \right)_0$ and
$\left( \simik \right)_0$ denote the stabilizer of the $0$ in the groups $Z_2^{\mxi}$ and $\simik$, respectively.
\end{csop}

\begin{csop}{$\Aut(+_0,0,1)$}
%Az $\aut{+_0,0,1}$ csoport már előkerült $\aut{\Sigma_1}$ szemidirekt szorzatként való jellemzésénél. Az $1$-et
%fixáló, vagy ekvivalensen a komplementer műveletét megőrző lineáris permutációk csoportja.
The group $\Aut(+_0,0,1)$ was briefly discussed in the section about $\Aut(\Sigma_1)$. It is the group of linear permutations
which preserve $1$ and the complementation.
\end{csop}

\begin{lemma}
%A kanonikus lineáris függvények automorfizmus-csoportjaiból metszetképzéssel megkapható csoportok az 1. Ábrán látható háló
%$[ \aut{+_0,0,1}, \aut {\emptyset} ]$ intervallumát alkotják a tartalmazásra, mint rendezésre nézve.
The automorphism groups of the canonical linear functions are ordered by inclusion as the $[ \Aut(+_0,0,1), \Aut(\emptyset) ]$ 
interval on Figure \ref{figure:func}.
\end{lemma}
\begin{proof}
%Először jellemezzük az adott intervallumban szereplő, eddig még nem leírt három csoportot:

%Megmutatjuk, hogy az $\aut{0},\aut{1},\aut{\Sigma}$ és $\aut{\neg}$ csoportok antiláncot alkotnak.
We will show that the groups $\Aut(0),\Aut(1),\Aut(\Sigma)$ and $\Aut(\neg)$ form an antichain.
\begin{itemize}
%\item Legyenek $a,b \in \ba \setminus \{0,1\}$ tetszőleges elemek, melyek nem egymás komplementerei. Ekkor az
%$(a,a+1,b,1)$ ciklus eleme $\aut{0}$-nak, jelöljük $\phi$-vel.
\item Let $a,b \in \ba \setminus \{0,1\}$ be arbitrary elements such that $a \neq \neg b$. Then the $4$-cycle $(a,a+1,b,1)$
is in $\aut{0}$. We will denote this permutation by $\phi$.
%Ekkor $\phi(1)=a \neq 1$ miatt $\aut{0} \nb \aut{1}$.
Then $\Aut(0) \nb \Aut(1)$ because $\phi(1)=a \neq 1$.
%$\phi(\neg a)=b \neq a=\neg \phi (a)$ miatt $\aut{0} \nb \aut{\neg}$.
$\Aut(0) \nb \Aut(\neg)$ because $\phi(\neg a)=b \neq a=\neg \phi (a)$.
%Illetve $\phi(\Sigma(a,a+1,1))=0 \neq a+b+1=\Sigma(\phi(a),\phi(a+1),\phi(1))$
%miatt $\aut{0} \nb \aut{\Sigma}$.
And $\Aut(0) \nb \Aut(\Sigma)$ because $\phi(\Sigma(a,a+1,1))=0 \neq a+b+1=\Sigma(\phi(a),\phi(a+1),\phi(1))$.
%\item Hasonlóan bizonyítható, hogy $\aut{1} \nb \aut{0}$,$\aut{1} \nb \aut{\neg}$ és $\aut{1} \nb \aut{\Sigma}$.
\item Similarly $\Aut(1) \nb \Aut(0)$,$\Aut(1) \nb \Aut(\neg)$ and $\Aut(1) \nb \Aut(\Sigma)$.

%\item Most legyenek $a,b \in \ba \setminus \{0,1\}$ tetszőleges elemek, melyek nem egymás komplementerei. Ekkor az
%$(a,b,0,a+1,b+1,1)$ ciklus eleme $\aut{\neg}$-nak, jelöljük $\phi$-vel. Mivel $\phi$ sem a $0$-t, sem  az
%$1$-et nem fixálja, így $\aut{\neg} \nb \aut{0}$ és $\aut{\neg} \nb \aut{1}$. Továbbá
%$\phi(\Sigma(a,b,0))=a+b \neq a+b+1=\Sigma(\phi(a),\phi(b),\phi(0))$ miatt $\aut{\neg} \nb \aut{\Sigma}$.
\item  Let $a,b \in \ba \setminus \{0,1\}$ be arbitrary elements such that $a \neq \neg b$. Then the $6$-cycle 
$(a,b,0,a+1,b+1,1)$ is in $\Aut(\neg)$. We will denote this permutation by $\phi$.
$\Aut(\neg) \nb \Aut(0)$ and $\Aut(\neg) \nb \Aut(1)$ because $\phi$ does not fix $0$ and $1$. Finally $\Aut(\neg) \nb \Aut(\Sigma)$
because $\phi(\Sigma(a,b,0))=a+b \neq a+b+1=\Sigma(\phi(a),\phi(b),\phi(0))$

%\item Legyen $\tau_a(x)=x+a$ tetszőleges eltolás. Ekkor $\tau_a$ megőrzi $\Sigma$-t, viszont $a \neq 0$ esetén
%$\tau_a$ nem fixálja sem a $0$-t, sem az $1$-et, így $\aut{\Sigma} \nb \aut{0}$ és $\aut{\Sigma} \nb \aut{1}$.
%Legyen $\phi$ olyan $ \on{GL}(\infty,2)$-beli permutáció,
%amely nem fixálja az $1$-et, és így nem tartja a $\neg$ műveletet, mivel $ \on{GL}(\infty,2)$ minden eleme fixálja a
%$0$-t.
%Mivel $\aut{+_0}=\on{GL}(\infty,2) \subset \aut{\Sigma}$, így $\aut{\Sigma} \nb \aut{\neg}$. 
\item 
Let $\tau_a(x)=x+a$ be an arbitrary translation. Then $\tau_a$ preserves $\Sigma$ but if $a \neq 0$ then
$\tau_a$ does not preserve $0$ or $1$. So $\Aut(\Sigma) \nb \Aut(0)$ and $\Aut(\Sigma) \nb \Aut(1)$.
Let $\phi$ be a permutation from $ \on{GL}(\infty,2)$ such that $\phi$ does not fix $1$. Then $\phi$ does not preserve
$\neg$ so $\Aut(\Sigma) \nb \Aut(\neg)$ because $\Aut(+_0)=\on{GL}(\infty,2) \subset \Aut(\Sigma)$.
\end{itemize}

%Most az $\aut{0},\aut{1},\aut{\Sigma},\aut{\neg}$ antilánc elemeiből képezhető metszeteket fogjuk meghatározni.
%A páronkénti metszetek:
We will determine the possible intersections of $\Aut(0),\Aut(1),\Aut(\Sigma)$ and $\Aut(\neg)$.
The intersections of every possible pair:
\begin{itemize}
%\item $\aut{0} \cap \aut{1}=\aut{0,1}$ a stabilizátorok definíciója alapján.
\item $\Aut(0) \cap \Aut(1)=\Aut(0,1)$ by the definition of stabilizers.
%\item $\aut{0} \cap \aut{\Sigma}=\aut{+_0}=\on{GL}(\infty,2)$, \aref{omkat2}. Következményt használva a
%$+_0(x,y)=\Sigma(x,y,0)$, $0=+_0(x,x)$ és $\Sigma(x,y,z)=+_0(x,+_0(y,x))$ formulák miatt.
\item $\Aut(0) \cap \Aut(\Sigma)=\Aut(+_0)=\on{GL}(\infty,2)$ because
$+_0(x,y)=\Sigma(x,y,0)$, $0=+_0(x,x)$ and $\Sigma(x,y,z)=+_0(x,+_0(y,x))$,  using Corollary \aref{omkat2}
\item $\Aut(0) \cap \Aut(\neg)= \Aut(\neg, 0, 1)$ 
%\item $\aut{1} \cap \aut{\Sigma}=\aut{+_1}=\on{GL}^1(\infty,2)$ \aref{omkat2}. Következményt használva a
%$+_1(x,y)=\Sigma(x,y,1)$, $1=+_1(x,x)$ és $\Sigma(x,y,z)=+_1(x,+_1(y,x))$ formulák miatt.
\item $\Aut(1) \cap \Aut(\Sigma)=\Aut(+_1)=\on{GL}^1(\infty,2)$ because
$+_1(x,y)=\Sigma(x,y,1)$, $1=+_1(x,x)$ and $\Sigma(x,y,z)=+_1(x,+_1(y,x))$ using Corollary \aref{omkat2}
%\item $\aut{0} \cap \aut{\neg}= \aut{\neg , 0, 1}$ 
\item $\Aut(0) \cap \Aut(\neg)= \Aut(\neg, 0, 1)$ 
%\item $\aut{\Sigma} \cap \aut{\neg}=\aut{\Sigma_1}$ \aref{omkat2}. Következményt használva a
%$\Sigma(x,y,z)=\Sigma_1(x,y,\Sigma_1(z,z,z))$, $\neg x=\Sigma_1(x,x,x)$ és
%$\Sigma_1(x,y,z)=\Sigma(x,y,\neg z)$ formulák miatt.
\item $\Aut(\Sigma) \cap \Aut(\neg)=\Aut(\Sigma_1)$ because
$\Sigma(x,y,z)=\Sigma_1(x,y,\Sigma_1(z,z,z))$, $\neg x=\Sigma_1(x,x,x)$ and
$\Sigma_1(x,y,z)=\Sigma(x,y,\neg z)$ using Corollary \aref{omkat2}.
\end{itemize}
%A hármas metszetek:
The possible triple intersections:
\begin{itemize}
%\item $\aut{0} \cap \aut{1} \cap \aut{\Sigma}=\aut{+_0,1}$ \aref{omkat2}. Következményt használva a
%$0=+_0(x,x)$, $1=1$, $\Sigma(x,y,z)=+_0(x,+_0(y,x))$ illetve $+_0(x,y)=\Sigma(x,y,0)$
%és $1=1$ formulák miatt.
\item 
$\Aut(0) \cap \Aut(1) \cap \Aut(\Sigma)=\Aut(+_0,1)$ because $0=+_0(x,x)$, $1=1$, $\Sigma(x,y,z)=+_0(x,+_0(y,x))$ and
$+_0(x,y)=\Sigma(x,y,0)$, $1=1$ using Corollary \aref{omkat2}.

\item $\Aut(0) \cap \Aut(1) \cap \Aut(\neg)=\Aut(\neg, 0, 1)$ 

%\item $\aut{0} \cap \aut{\Sigma} \cap \aut{\neg}=\aut{+_0,1}$ \aref{omkat2}. Következményt használva a
%$0=+_0(x,x)$, $\Sigma(x,y,z)=+_0(x,+_0(y,x))$, $\neg x=+_0(x,1)$ illetve $+_0(x,y)=\Sigma(x,y,0)$
%és $\neg 0=1$ formulák miatt.
\item $\Aut(0) \cap \Aut(\Sigma) \cap \Aut(\neg)=\Aut(+_0,1)$ because $0=+_0(x,x)$, $\Sigma(x,y,z)=+_0(x,+_0(y,x))$, $\neg x=+_0(x,1)$
and $+_0(x,y)=\Sigma(x,y,0)$, $\neg 0=1$ using Corollary \aref{omkat2}.

%\item $\aut{1} \cap \aut{\Sigma} \cap \aut{\neg}=\aut{+_0,1}$ \aref{omkat2}. Következményt használva a
%$1=1$, $\Sigma(x,y,z)=+_0(x,+_0(y,x))$, $\neg x=+_0(x,1)$ illetve $+_0(x,y)=\Sigma(x,y,\neg 1)$
%és $1=1$ formulák miatt.
\item $\Aut(1) \cap \Aut(\Sigma) \cap \Aut(\neg)=\Aut(+_0,1)$ because $1=1$, $\Sigma(x,y,z)=+_0(x,+_0(y,x))$,
$\neg x=+_0(x,1)$ and $+_0(x,y)=\Sigma(x,y,\neg 1)$, $1=1$ using Corollary \aref{omkat2}.
\end{itemize}

%Végül az antilánc mind a négy elemének metszete:
The intersection of all four elements of the antichain:
\begin{itemize}
%\item $\aut{0} \cap \aut{1} \cap \aut{\Sigma} \cap \aut{\neg}=\aut{+_0,1}$ \aref{omkat2}. Következményt használva a
%$0=+_0(x,x)$, $1=1$, $\Sigma(x,y,z)=+_0(x,+_0(y,x))$, $\neg x=+_0(x,1)$ illetve $+_0(x,y)=\Sigma(x,y,0)$
%és $\neg 0=1$ formulák miatt.
\item $\Aut(0) \cap \Aut(1) \cap \Aut(\Sigma) \cap \Aut(\neg)=\Aut(+_0,1)$ because
$0=+_0(x,x)$, $1=1$, $\Sigma(x,y,z)=+_0(x,+_0(y,x))$, $\neg x=+_0(x,1)$ and $+_0(x,y)=\Sigma(x,y,0)$, $\neg 0=1$
using Corollary \aref{omkat2}.
\end{itemize}
%Ennek alapján minden kanonikus lineáris függvény automorfizmus-csoportja előáll, mint az
%$\aut{0},\aut{1},\aut{\Sigma},\aut{\neg}$ antilánc valahány elemének metszete.
%Az $\aut{x}=\sym$ szimmetrikus csoport az üres metszetnek felel meg.
So we can conclude that the automorphism group of any canonical linear function can be obtained as the intersection of some of
the groups $\Aut(0),\Aut(1),\Aut(\Sigma)$ and $\Aut(\neg)$. The group $\Aut(x)=\sym$ corresponds to the empty intersection.

%A kanonikus lineáris függvények automorfizmuscsoportjain kívül még három csoport kapható meg metszetként. Ezek
%$\aut{0,1}=\on{Sym}_{(0,1)}(\ba )$, ami a $\{ 0,1 \}$ halmaz pontonkénti stabilizátora a $\sym$ csoportban, a
%$\aut{\neg,0,1}$ csoport, illetve
%$\aut{+_0,1}$, ami az egy konstanssal ellátott vektortér automorfizmus-csoportja:
%$\on{GL}(\infty,2) \cap \on{\Sym}_1(\ba)$. Ezek felsorolásunkban a 9., 10. és 11. csoportok.
There are three other groups which can be obtained as the intersection of some coatoms:
\begin{itemize}
\item The group $\Aut(0,1)=\on{Sym}_{(0,1)}(\ba)$.
\item The group $\Aut(\neg,0,1)$.
\item The group $\Aut(+_0,1)$.
\end{itemize}
These are the groups of number 9, 10 and 11.
\end{proof}

%\section{A megszámlálható atommentes Boole-algebra funkcionális reduktjai}
\section{The functional reducts of the countable atomless Boolean algebra}

%Ebben a fejezetben befejezzük a funkcionális reduktok kölcsönös definiálhatóság erejéig való klasszifikálását.
In this section we will finish the classification of the functional reducts.
%Az előző fejezetben meghatároztuk, hogy valahány $l_1, l_2 \ldots$ lineáris kifejezésfüggvényt megőrző
%permutációk milyen zárt részcsoportokat alkothatnak, ezek a zárt részcsoportok 1. Ábrán látható háló
%$[\{ +,0,1 \} , \emptyset ]$ intervallumát alkotják a tartalmazásra mint rendezésre nézve.
%Ezt a hálót fogjuk most kiegészíteni a nemlineáris kifejezésfüggvények által
%meghatározott lehetséges automorfizmusokkal.
In the previous section we finished the case of the linear functional reducts, the remaining case is the non-linear one.

\begin{theorem}\label{halometszet}
%A $\ba$ funkcionális reduktjai 1. Ábrán látható hálót alkotják, ahol
%a rendezés az automorfizmus-csoportok tartalmazása.
The possible automorphism groups of functional reducts are ordered by inclusion as in Figure \ref{figure:func}.
\end{theorem}

\begin{proof}
%\Aref{nemlinear}. Lemma alapján egy $f$ nemlineáris kifejezésfüggvényre vagy $\aut{f}=\aba$, vagy $\aut{f}=\aut{M}$.
%Mivel tudjuk, hogy $\aba$ a háló minimális eleme, így csak $\aut{M}$ helyét kell meghatároznunk a hálóban.
%$T \leq \aut{M}$ miatt $\aut{M} \nb \aut{0}$ és $\aut{M} \nb \aut{1}$, így annyit kell még belátnunk, hogy
%$\aut{M} \lneq \aut{\Sigma_1}$.
If $f$ is a non-linear term function then  $\Aut(f)=\aba$ or $\Aut(f)=\Aut(M)$ by Lemma \ref{nemlinear}.
The group $\aba$ is the minimal element of the lattice. We need to find the place of $\Aut(M)$.
$\Aut(M) \nb \Aut(0)$ and $\Aut(M) \nb \Aut(1)$ because $T \leq \Aut(M)$. We will show that $\Aut(M) \leq \Aut(\Sigma_1)$.
%\Aref{msemidir}. Tétel miatt $\aut{M} \leq \aut{\Sigma_1}$ bizonyításához elég belátnunk, hogy az eltolások
%megőrzik $\Sigma_1$-et. Legyen $\tau_a(x)=x+a$ tetszőleges eltolás, ekkor
%$\tau_a(\Sigma_1(x,y,z))=x+y+z+1+a=x+a+y+a+z+a+1=\Sigma_1(\tau_a(x),\tau_a(y),\tau_a(z))$, tehát az eltolások
%megőrzik a $\Sigma_1$ műveletet.
 By Theorem \ref{msemidir} 
it is enough to show that translations preserve $\Sigma_1$. Let $\tau_a(x)=x+a$ be arbitrary translation then
$$\tau_a(\Sigma_1(x,y,z))=x+y+z+1+a=x+a+y+a+z+a+1=$$
$$=\Sigma_1(\tau_a(x),\tau_a(y),\tau_a(z))$$
so translations preserve $\Sigma_1$.

In order to finish our proof it is enough to show that $\Aut(\ba)\neq \Aut(+_0,0,1)$ and $\Aut(M)\neq \Aut(\Sigma_1)$.
%In order to prove the strict containment we need the existence of a permutation such that it preserves $\Sigma_1$ but
%does not preserve the median. A permutation which fixes $0$ and $1$ and preserves $+_0$ but is not a Boolean algebra automorphism
%is sufficient. 
Let $a,b \in \ba$ two elements such that $0<a<b<1$. In this case $a,b$ and $1$ are linearly independent in the vector space structure,
so there is a permutation $\phi$ in the linear group which fixes the $1$ and switches $a$ and $b$. This permutation cannot
be a Boolean algebra automorphism because $a<b$ and $\phi(a)>\phi(b)$. This implies $\Aut(\ba)\neq \Aut(+_0,0,1)$.
The other statement follows from the fact $\Aut(M)\cap \Aut(0)=\Aut(\ba)$ and $\Aut(\Sigma_1)\cap \Aut(0)=\Aut(+_0,0,1)$.
\end{proof}

%Tehát tudjuk, hogy A $\ba$ funkcionális reduktjai az 1. Ábrán látható hálót alkotják, ahol
%a rendezés az automorfizmus-csoportok tartalmazása. Ennél azonban több is igaz:
So we proved that the possible automorphism groups of functional reducts are exactly as in Figure \ref{figure:func}. 

Regarding the place of functional reducts amongst all reducts, all of the reducts currently known to us have the same algebraic closure operator as some functional reduct. Therefore the following problems arise naturally:

\begin{problem}\label{connectionProblem}
Find a connection between functional and (relational) first-order definable reducts. Is it true that the (model theoretic) algebraic closure operation of any first-order definable reduct is always that of a functional reduct?
\end{problem}

\begin{problem}\label{allReductsProblem}
Find all first-order definable reducts of the homogeneous  Boolean algebra. %Is it true, that there are only finitely many of them?
\end{problem}

%\begin{problem}
%Find the functional reducts of finite Boolean algebras up to first-order interdefinability.
%\end{problem}
 
	We give a list of those reducts of the countable atomless Boolean algebra that we managed to identify. Our list contains 12 infinite ascending chains, and 23 individual reducts. Note that this list might be incomplete, we only provide it to show how our current result might be extended. This list is a complete sublattice of the lattice of the supergroups: it is closed under intersections and under taking the smallest closed group containing any given subset of the list. 

	First we need the following definitions.

\begin{definition}\label{independence}
Two elements $a, b \in \ba$ are independent if they generate a $16$-element Boolean algebra.
\end{definition}

	The following definition is from \cite{vegtelen}.

\begin{definition}\label{inf_chain}
	Let $W_n$ be an $n$-codimensional subspace of $\ba$ containing $1$. Then $h_n\in \sym$ is defined by $h_n(v)=v+1$ for $v\in W_n$, and $h_n(v)=v$ otherwise.

	We define $H_n$ to be the closure of the group $\langle \Aut(+_0,0,1), h_n\rangle$. 
\end{definition}

Since $\Aut(+_0,0,1)$ acts transitively on subspaces of codimension $n$, we see that the group $H_n$ does not depend on our particular choice of $W_n$, only on the codimension $n$. By Proposition 6 in \cite{vegtelen} we know that the groups $H_1,H_2,\dots$ form an infinite ascending chain.

	Now we are ready to give our list of reducts. We only label those reducts here which are not functional reducts. We will describe the reduct, its automorphism group, or both, without any proofs. There are 12 individual reducts which are not functional reducts. There are indexed by lowercase letters (a)-(l). The reducts in the 12 infinite ascending chains are labelled by a lowercase letter (m)-(x), and a natural number indicating their place in the chain. Two of the reducts in these chains are also functional reducts. The automorphism group of a reduct labelled $x$ will be denoted by $G_x$.
% where $x$ is the corresponding lowercase letter.

\begin{enumerate}[(a)]
\item Let $f$ denote the permutation $f(0)=0, f(1)= 1$ and  $f(x)=x+1$ for all other elements. Then $G_a = \left< \aba, f \right>$. The group $\aba$ is a subgroup of index $2$ in $G_a$.
\item Let $g$ denote the permutation $g(0)=1, g(1)= 0$ and  $g(x)=x$ for all other elements. Then $G_b = \left< \aba, g \right>$. The group $\aba$ is a subgroup of index $2$ in $G_b$.
\item Let $h$ denote the permutation $h(x)=x+1$. Then $G_c = \left< \aba, h \right>$. The group $\aba$ is a subgroup of index $2$ in $G_c$.
\item Let $f, g, h$ denote permutations as in the definitions of the reducts (a), (b) and (c). Note that $fg=h$. Then
$G_d = \left< \aba, f, g \right>$. The group $\aba$ is a subgroup of index $4$ in $G_d$.
\item This reduct can be characterized by the the independence relation (Definition \ref{independence}). The automorphism group $G_e$ consists of the permutations of the form $g_1g_2$ where $g_1 \in \aba$ and $g_2(x) \in \{x, x+1\}$ for every $x \in \ba$.
\item The group $G_f$ can be obtained as $G_f = G_e\cap \Aut(0,1)$.
\item The group $G_g$ can be obtained as $G_g = \cl{ \Aut(M), G_e }$.
\item The group $G_h$ can be obtained as $G_h = \cl{ \Aut(+_0, 0, 1), G_e }$.
\item The group $G_i$ can be obtained as $G_i = \cl{ \Aut(+_0, 0, 1), G_f }$.
\item The group $G_j$ can be obtained as $G_j = \cl{ \Aut(\Sigma_1), G_e }$.
\item The group $G_k$ is the group of permutations fixing the set $\{0,1\}$ setwise.
\item The group $G_l$ is the subgroup of $\Aut(\neg)$ fixing the set $\{0,1\}$ setwise.
\item We put $G_{m_i}:=H_i\cap \Aut(0)$ (see Definition \ref{inf_chain}). Then $G_{m_1}=H_1\cap \Aut(0)=\Aut(+_0,0,1)$.
\item Let $f$ be the permutation defined in the definition of $G_a$. Then
$G_{n_i} = \left < G_{m_i}, f \right>$.
\item Let $g$ be the permutation defined in the definition of $G_b$. Then
$G_{o_i} = \left < G_{m_i}, g \right>$.
\item Let $h$ be the permutation defined in the definition of $G_c$. Then
$G_{p_i} = \left < G_{m_i}, h \right>=H_i$.
\item Let $f,g$ be the permutation defined in the definition of $G_a$ and $G_b$. Then
$G_{q_i} = \left < G_{m_i}, f,g \right>$.
\item The group $G_{r_i}$ can be obtained as $G_{r_i}=\cl{\Aut(M),G_{m_i}}$. Then $G_{r_1}=\Aut(\Sigma_1)$.
\item The group $G_{s_i}$ can be obtained as $G_{s_i}=G_{m_i}\cap G_e$.
\item Let $f$ be the permutation defined in the definition of $G_a$. Then
$G_{t_i} = \left < G_{s_i}, f \right>$.
\item Let $g$ be the permutation defined in the definition of $G_b$. Then
$G_{u_i} = \left < G_{s_i}, g \right>$.
\item Let $h$ be the permutation defined in the definition of $G_c$. Then
$G_{v_i} = \left < G_{s_i}, h \right>$.
\item Let $f,g$ be the permutation defined in the definition of $G_a$ and $G_b$. Then
$G_{w_i} = \left < G_{s_i}, f,g \right>$.
\item The group $G_{x_i}$ can be obtained as $G_{x_i}=\cl{\Aut(M),G_{s_i}}$.
\end{enumerate}
We hope that the above list is close to complete, and that this classification of functional reducts will help for classifying all reducts.

\begin{figure}\label{frhalo2}
\centering
\scalebox{.8}{\begin{tikzpicture}[yscale=0.31,xscale=0.5]
	\begin{pgfonlayer}{nodelayer}
		\node [style={light_blue3_nf}, label=center:\mysize{$c$}] (1) at (-1.25, 3.25) {};
		\node [style={light_blue3_nf}, label=center:\mysize{$b$}] (2) at (-3.25, 3.25) {};
		\node [style={light_blue3_nf}, label=center:\mysize{$a$}] (3) at (-5.25, 3.25) {};
		\node [style={light_blue3_f}] (4) at (-5.25, 0) {$\cdot,+_0,0,1$};
		\node [style={light_blue3_nf}, label=center:\mysize{$d$}] (5) at (-4.5, 6.25) {};
		\node [style={dark_blue3_f}] (6) at (-1.25, 7.25) {$M$};
		\node [style={dark_green2_nf}, label=center:\mysize{$v_1$}] (7) at (-7.75, 6.25) {};
		\node [style={dark_green2_nf}, label=center:\mysize{$u_1$}] (8) at (-9.75, 6.25) {};
		\node [style={dark_green2_nf}, label=center:\mysize{$t_1$}] (9) at (-11.75, 6.25) {};
		\node [style={dark_green1_nf}, label=center:\mysize{$s_1$}] (10) at (-13.75, 3.25) {};
		\node [style={dark_green2_nf}, label=center:\mysize{$w_1$}] (11) at (-11, 9.25) {};
		\node [style={cornflower_blue_nf}, label=center:\mysize{$x_1$}] (22) at (-5.25, 9.25) {};
		\node [style={dark_yellow2_nf}, label=center:\mysize{$p_1$}] (23) at (-17.25, 10.25) {};
		\node [style={dark_yellow2_nf}, label=center:\mysize{$o_1$}] (24) at (-19.25, 10.25) {};
		\node [style={dark_yellow2_nf}, label=center:\mysize{$n_1$}] (25) at (-21.25, 10.25) {};
		\node [style={green_f}] (26) at (-23.25, 7.25) {$+_0,0,1$};
		\node [style={dark_yellow2_nf}, label=center:\mysize{$q_1$}] (27) at (-20.25, 13.25) {};
		\node [style={dark_cornflower_blue2_f}] (28) at (-15.25, 13.25) {$\Sigma_1$};
		\node [style={dark_green2_nf}, label=center:\mysize{$v_2$}] (29) at (-7.75, 16.75) {};
		\node [style={dark_green2_nf}, label=center:\mysize{$u_2$}] (30) at (-9.75, 16.75) {};
		\node [style={dark_green2_nf}, label=center:\mysize{$t_2$}] (31) at (-11.75, 16.75) {};
		\node [style={dark_green1_nf}, label=center:\mysize{$s_2$}] (32) at (-13.75, 13.75) {};
		\node [style={dark_green2_nf}, label=center:\mysize{$w_2$}] (33) at (-11, 19.75) {};
		\node [style={dark_yellow2_nf}, label=center:\mysize{$p_2$}] (34) at (-17.25, 20.75) {};
		\node [style={dark_yellow2_nf}, label=center:\mysize{$o_2$}] (35) at (-19.25, 20.75) {};
		\node [style={dark_yellow2_nf}, label=center:\mysize{$n_2$}] (36) at (-21.25, 20.75) {};
		\node [style={dark_yellow1}, label=center:\mysize{$m_2$}] (37) at (-23.25, 17.75) {};
		\node [style={dark_yellow2_nf}, label=center:\mysize{$q_2$}] (38) at (-20.25, 23.75) {};
		\node [style={dark_cornflower_blue2_nf}, label=center:\mysize{$r_2$}] (39) at (-15.25, 23.75) {};
		\node [style={dark_green2_nf}, label=center:\mysize{$v_3$}] (40) at (-7.75, 27.25) {};
		\node [style={dark_green2_nf}, label=center:\mysize{$u_3$}] (41) at (-9.75, 27.25) {};
		\node [style={dark_green2_nf}, label=center:\mysize{$t_3$}] (42) at (-11.75, 27.25) {};
		\node [style={dark_green1_nf}, label=center:\mysize{$s_3$}] (43) at (-13.75, 24.25) {};
		\node [style={dark_green2_nf}, label=center:\mysize{$w_3$}] (44) at (-11, 30.25) {};
		\node [style={dark_yellow2_nf}, label=center:\mysize{$p_3$}] (45) at (-17.25, 31.25) {};
		\node [style={dark_yellow2_nf}, label=center:\mysize{$o_3$}] (46) at (-19.25, 31.25) {};
		\node [style={dark_yellow2_nf}, label=center:\mysize{$n_3$}] (47) at (-21.25, 31.25) {};
		\node [style={dark_yellow1}, label=center:\mysize{$m_3$}] (48) at (-23.25, 28.25) {};
		\node [style={dark_yellow2_nf}, label=center:\mysize{$q_3$}] (49) at (-20.25, 34.25) {};
		\node [style={dark_cornflower_blue2_nf}, label=center:\mysize{$r_3$}] (50) at (-15.25, 34.25) {};
		\node [style={cornflower_blue_nf}, label=center:\mysize{$x_2$}] (52) at (-5.25, 19.75) {};
		\node [style={cornflower_blue_nf}, label=center:\mysize{$x_3$}] (54) at (-5.25, 30.25) {};
		\node [style={magenta_nf}] (55) at (-13.75, 54.25) {$f$};
		\node [style={dark_orange2_nf}] (56) at (-23.25, 58.25) {$i$};
		\node [style={magenta_nf}] (57) at (-11, 61.25) {$e$};
		\node [style={dark_orange2_nf}] (58) at (-20.5, 65.25) {$h$};
		\node [style=new style 0] (59) at (-5.25, 67.25) {$g$};
		\node [style={dark_orange2_nf}] (61) at (-14.75, 71.25) {$j$};
		\node [style={green_f}] (62) at (-27.25, 63.25) {$+_1,1$};
		\node [style={green_f}] (63) at (-32.25, 63.25) {$+_0,0$};
		\node [style={green_f}] (64) at (-20, 76.25) {$\Sigma$};
		\node [style={red_f}] (65) at (-6.25, 71.75) {$\neg,0,1$};
		\node [style={red_nf}] (66) at (-6.25, 75.5) {$l$};
		\node [style={red_f}] (67) at (-6.25, 77.5) {$\neg$};
		\node [style={red_f}] (68) at (-14.75, 75.75) {0,1};
		\node [style={red_f}] (69) at (-17.75, 78.25) {0};
		\node [style={red_f}] (70) at (-14.75, 78.25) {1};
		\node [style={red_nf}] (71) at (-11.75, 78.25) {$k$};
		\node [style={red_f}] (72) at (-16.25, 84.5) {$\emptyset$};
	\end{pgfonlayer}
	\begin{pgfonlayer}{edgelayer}
		\draw (4) to (10);
		\draw (4) to (3);
		\draw (4) to (2);
		\draw (4) to (1);
		\draw (5) to (3);
		\draw (5) to (2);
		\draw (5) to (1);
		\draw (1) to (6);
		\draw (9) to (10);
		\draw (8) to (10);
		\draw (7) to (10);
		\draw (11) to (9);
		\draw (11) to (8);
		\draw (11) to (7);
		\draw (22) to (6);
		\draw (22) to (7);
		\draw (25) to (26);
		\draw (24) to (26);
		\draw (23) to (26);
		\draw (27) to (25);
		\draw (27) to (24);
		\draw (27) to (23);
		\draw (26) to (10);
		\draw (25) to (9);
		\draw (24) to (8);
		\draw (23) to (7);
		\draw (22) to (28);
		\draw (23) to (28);
		\draw (31) to (32);
		\draw (30) to (32);
		\draw (29) to (32);
		\draw (33) to (31);
		\draw (33) to (30);
		\draw (33) to (29);
		\draw (36) to (37);
		\draw (35) to (37);
		\draw (34) to (37);
		\draw (38) to (36);
		\draw (38) to (35);
		\draw (38) to (34);
		\draw (37) to (32);
		\draw (36) to (31);
		\draw (35) to (30);
		\draw (34) to (29);
		\draw (34) to (39);
		\draw (42) to (43);
		\draw (41) to (43);
		\draw (40) to (43);
		\draw (44) to (42);
		\draw (44) to (41);
		\draw (44) to (40);
		\draw (47) to (48);
		\draw (46) to (48);
		\draw (45) to (48);
		\draw (49) to (47);
		\draw (49) to (46);
		\draw (49) to (45);
		\draw (48) to (43);
		\draw (47) to (42);
		\draw (46) to (41);
		\draw (45) to (40);
		\draw (45) to (50);
		\draw (54) to (40);
		\draw (52) to (29);
		\draw (52) to (22);
		\draw (54) to (52);
		\draw (37) to (26);
		\draw (25) to (36);
		\draw (24) to (35);
		\draw (23) to (34);
		\draw (10) to (32);
		\draw (27) to (38);
		\draw (9) to (31);
		\draw (8) to (30);
		\draw (7) to (29);
		\draw (37) to (48);
		\draw (36) to (47);
		\draw (35) to (46);
		\draw (34) to (45);
		\draw (31) to (42);
		\draw (32) to (43);
		\draw (52) to (39);
		\draw (54) to (50);
		\draw (30) to (41);
		\draw (29) to (40);
		\draw (11) to (27);
		\draw (33) to (38);
		\draw (44) to (49);
		\draw (56) to (55);
		\draw (58) to (57);
		\draw [style=dashed] (56) to (48);
		\draw [style=dashed] (56) to (47);
		\draw [style=dashed] (58) to (46);
		\draw [style=dashed] (58) to (49);
		\draw [style=dashed] (58) to (45);
		\draw (58) to (56);
		\draw (57) to (55);
		\draw [style=dashed] (55) to (43);
		\draw [style=dashed] (55) to (42);
		\draw [style=dashed] (57) to (44);
		\draw [style=dashed] (57) to (41);
		\draw [style=dashed] (57) to (40);
		\draw [style=dashed] (59) to (54);
		\draw (57) to (59);
		\draw (61) to (58);
		\draw [style=dashed] (61) to (50);
		\draw (59) to (61);
		\draw (26) to (62);
		\draw (26) to (63);
		\draw (64) to (63);
		\draw (64) to (62);
		\draw (64) to (28);
		\draw (39) to (28);
		\draw (39) to (50);
		\draw (66) to (65);
		\draw (67) to (66);
		\draw (65) to (56);
		\draw (66) to (58);
		\draw (67) to (61);
		\draw (65) to (68);
		\draw (68) to (69);
		\draw (68) to (70);
		\draw (68) to (71);
		\draw (66) to (71);
		\draw (64) to (72);
		\draw (69) to (72);
		\draw (72) to (70);
		\draw (71) to (72);
		\draw (67) to (72);
		\draw (33) to (11);
		\draw (38) to (49);
		\draw (33) to (44);
		\draw (5) to (11);
		\draw (3) to (9);
		\draw (8) to (2);
		\draw (7) to (1);
	\end{pgfonlayer}
\end{tikzpicture}}
\caption{\label{figure:all}The lattice of all reducts known to us, ordered by the inclusion of their automorphism groups.}
\end{figure}


\newpage

\begin{thebibliography}{99}

%\bibitem{lamport94}
  %Leslie Lamport,
  %\emph{\LaTeX: a document preparation system}.
  %Addison Wesley, Massachusetts,
  %2nd edition,
  %1994.

\bibitem{velpos}
\textsc{N. Ackerman, C. Freer, R. Patel}.
\emph{Invariant measures concentrated on countable structures}.
Forum of Mathematics, Sigma. Vol. 4. Cambridge University Press, 2016.

\bibitem{benn}
\textsc{J. H. Bennett}.
\emph{The reducts of some infinite homogeneous graphs and tournaments}.
{Rutgers university}, Doctoral Thesis,
1997.

\bibitem{nk}
\textsc{M. Bodirsky, M. Pinsker, A. Pongr\'{a}cz}.
\emph{The 42 reducts of the
random ordered graph}.
Proceedings of the London Mathematical Society 
111(3): 591-632,
2015

\bibitem{rams}
\textsc{M. Bodirsky, M. Pinsker}.
\emph{Reducts of Ramsey Structures},
Model Theoretic Methods in Finite Combinatorics,
American Mathematical Society,
Contemporary Mathematics,
558,
489-519,
2011.

\bibitem{vegtelen}
\textsc{B. Bodor, P. J. Cameron, Cs. Szab\'{o}}
\emph{Infinitely many reducts of homogeneous structures}, 
Algebra Universalis,
79(2), 43,
2018.

\bibitem{burris}
\textsc{S. Burris, H. P. Sankappanavar}
\emph{A course in universal algebra},
Graduate Texts Math 78,
1981.


\bibitem{cam_op}
	\textsc{P. J. Cameron}.
	\emph{Oligomorphic permutation groups}.
	London Mathematical Society Lecture Note Series, 152. Cambridge University Press, Cambridge, 1990.

\bibitem{camren}
\textsc{P. J. Cameron}.
\emph{Transitivity of permutation groups on unordered sets}.
Mathematische Zeitschrift,
148:
127-139,
1976.

\bibitem{frass}
	\textsc{R. Fra\"iss\'{e}}.
	\emph{Sur certaines relations qui g\'{e}n\'{e}ralisent l’order des nombres rationnels}. 
    Comptes rendus de l’Acad\'{e}mie des Sciences de Paris 237: 540–542, 1953.

\bibitem{henson}
\textsc{C. W. Henson}.
\emph{A family of countable homogeneous graphs}.
Pacific Journal of Mathematics,
38:
69-83,
1971.


\bibitem{hodges}
	\textsc{W. Hodges}.
	\emph{Model theory}.
	Encyclopedia of Mathematics and its Applications, 42. Cambridge University Press, Cambridge, 1993.


\bibitem{juzi}
\textsc{M. Junker, M. Ziegler}.
\emph{The 116 reducts of $(\mathbb{Q},<,a)$}.
J. Symbolic Logic 73(3): 861–884, 2008.

\bibitem{macpherson}
        \textsc{D. Macpherson}.
        \emph{A survey of homogeneous structures}.
        Discrete Mathematics, 311(15): 1599–1634, 2011.


\bibitem{pppsz}
       \textsc{P. P. Pach, M. Pinsker, G. Pluh\'{a}r, A. Pongr\'{a}cz, Cs. Szab\'{o}}.
		\emph{Reducts of the random partial order}.
		Advances in Mathematics 267: 94-120. 2014.

\bibitem{pongi}
\textsc{A. Pongr\'{a}cz}.
\emph{Reducts of the Henson graphs with a constant}.
Annals of Pure and Applied Logic,
168(7): 1472-1489, 2017

\bibitem{thomas}
        \textsc{S. Thomas}.
        \emph{Reducts of the random graph}.
        Journal of Symbolic Logic, 56(1): 176–181, 1991.

\end{thebibliography}
\end{document}